\documentclass[11pt, a4paper, oneside]{amsart}

\usepackage{amsmath, amsthm, amsfonts, amssymb, mathrsfs}
\usepackage[all]{xy}
\usepackage{xcolor}
\usepackage{hyperref}
\usepackage{multirow}
\usepackage[hypcap=false,width=\columnwidth]{caption}
\usepackage{authblk}
\usepackage{tabularx}
\usepackage{float}
\usepackage[titletoc]{appendix}
\hypersetup{
colorlinks=true, 
linkcolor=blue, 
urlcolor=red, 
}
\usepackage{enumerate}
\usepackage[shortlabels]{enumitem}
\usepackage{enumitem} 
\usepackage{geometry}               
\usepackage[capitalise]{cleveref}

\newfont{\sheaf}{eusm10 scaled\magstep1}

\def\N{\ensuremath{\mathbb N}}

\def\C{\ensuremath{\mathbb C}}
\def\P{\ensuremath{\mathbb P}}
\newcommand{\F}{\mathcal{F}}
\newcommand{\UU}{\mathbb{U}}
\def\L{\ensuremath{\mathcal L}}

\def\L{\ensuremath{\mathbf L}}
\newcommand{\KK}{\mathcal{K}}
\newcommand{\T}{\mathcal{T}}
\newcommand{\X}{\mathfrak{X}}

\newcommand{\ann}{\mathrm{ann}}
\newcommand{\II}{\mathscr{C}}

\newcommand{\codim}{\mathrm{codim}}

\newcommand\CC{{\mathbb{C}}}
\newcommand\RR{{\mathbb{R}}}

\newcommand\ZZ{{\mathbb{Z}}}

\newcommand\NN{{\mathbb{N}}}
\newcommand\PP{{\mathbb{P}}}
\def\L{\ensuremath{\mathbf L}}
\def\P{\ensuremath{\mathbb P}}
\newcommand\ii{{\imath}}

\def\D{{\mathcal D}}

\def\K{{\mathcal K}}
\def\L{{\mathcal L}}
\def\O{{\mathcal O}}

\def\U{{\mathcal U}}
\def\V{{\mathcal V}}
\def\W{{\mathcal W}}
\def\NN{{\mathcal{N}}}

\newtheorem{theorem}{Theorem}[section]
\newtheorem*{theorema}{Theorem A}
\newtheorem*{theoremb}{Theorem B}

\newtheorem{proposition}[theorem]{Proposition}

\newtheorem{lemma}[theorem]{Lemma}
\newtheorem{definition}[theorem]{Definition}
\newtheorem{corollary}[theorem]{Corollary}

\theoremstyle{definition}
\newtheorem{example}[theorem]{Example}
\newtheorem{remark}[theorem]{Remark}

\numberwithin{equation}{section}

\title{Rational pullbacks of toric foliations}

\makeatletter
\g@addto@macro{\endabstract}{\@setabstract}
\newcommand{\authorfootnotes}{\renewcommand\thefootnote{\@fnsymbol\c@footnote}}
\makeatother

\newenvironment{acknowledgement}{\textbf{Acknowledgments}}


\begin{document}

	\begin{center}
	\LARGE\textbf{Rational pullbacks of toric foliations }\par \bigskip

\normalsize
\authorfootnotes
Javier Gargiulo Acea\footnote{The author was fully supported by CNPq, Brazil, the project Print - Institutional Internationalization Program - CAPES, Brazil and Instituto de Matemática Pura e Aplicada, Brazil.}\textsuperscript{1}, 
Ariel Molinuevo\footnote{The author was fully supported by Universidade Federal do Rio de Janeiro, Brazil.}\textsuperscript{2},
Sebasti\'an Velazquez\footnote{The author was fully supported by CNPq, Brazil and CONICET, Argentina.}\textsuperscript{3}

\

\footnotesize{
	\textsuperscript{1}Departamento de Matem\'atica Aplicada, Universidade Federal Fluminense, Brazil. \par
	\textsuperscript{2}Instituto de Matematica, Universidade Federal do Rio de Janeiro, Brazil. \par
	\textsuperscript{3}Instituto de Matem\'atica Pura e Aplicada, Brazil\par\bigskip}

\today
\end{center}

\let\thefootnote\relax
\let\thefootnote\relax

\begin{abstract}
This article is dedicated to the study of singular codimension $1$ foliations $\F$ on a simplicial complete toric variety $X$ and their pullbacks by dominant rational maps $\varphi:\P^n\dashrightarrow X$. First, we describe the singularities of $\F$ and $\varphi^*\F$ for a generic pair $(\varphi,\F)$. Then we show that the first order deformations of $\varphi^*\F$ arising from first order unfoldings are the families of the form $\varphi_\varepsilon^*\F$, where $\varphi_\varepsilon$ is a perturbation of $\varphi$. We also prove that the deformations of the form $\varphi^*\F_\varepsilon$ consist exactly of the families which are tangent to the fibers of $\varphi$. In order to do so, we state some results of independent interest regarding the Kupka singularities of these foliations.
\end{abstract}

\tableofcontents

\section{Introduction}

\

The study of first order perturbations of codimension one foliations has shown to be an effective tool for finding new irreducible components of the space of codimension one foliations in $\PP^n$, see the works of \cite{CPV,fj,CGM}. Also the study of foliations in normal varieties has been gaining growing interest in the area, see for instance \cite{AD,AD2}. In this article we study first order perturbations, \emph{i.e.}, first order unfoldings and first order deformations of foliations on $\PP^n$ which are pullbacks of foliations on a normal toric  variety. Let us now introduce some notations and definitions to make a clearer exposition.

\

Let $X$ be a normal compact complex variety and $j:X_r\to X$ the inclusion of its smooth locus.
A singular codimension $1$ foliation $\F$ on $X$ can be defined as a foliation on $X_r$ and is induced by a unique element $[\omega]\in \PP H^0(X,j_*(\Omega^1_{X_r} \otimes \O_{X}(\D)\vert_{X_r}))$ for some element $\D\in Cl(X)$ satisfying the Frobenius integrability condition $\omega\wedge d\omega=0$ and whose zero locus $Sing(\omega)$ is of codimension greater than one. One can then define the moduli space of foliations of codimension one and degree $\D$ as the quasi-projective variety
\[
\F^1(X,\D)=\{[\omega]\in \P H^0(X,j_*(\Omega^1_{X_r} \otimes \O_{X}(\D)\vert_{X_r})): \omega \wedge d \omega =0  \mbox{ and } \codim(Sing(\omega))\geq 2\}.
\]
We will use $[\omega]$ or $\F_\omega$ indistinctly when no confusion arises. 
Let $\F_\omega$ be the foliation corresponding to an element $[\omega]\in \F^1(X,\D)$ and $X[\varepsilon]$ be the scheme $X\times \NN $ where $\NN=Spec(\C[\varepsilon]/(\varepsilon^2))$ is the scheme of \emph{dual numbers}. Let us also denote by $j[\varepsilon]:X_r[\varepsilon]\to X[\varepsilon]$ the natural inclusion and $\pi_1: X[\varepsilon]\to X$ the projection.  A \emph{first order deformation} of $\F_\omega$ is a family of foliations on $X$ parameterized by $\NN$. This is, an integrable element $\omega_\varepsilon\in H^0(X[\varepsilon], j[\varepsilon]_*(\Omega^1_{X_r[\varepsilon]|\NN}\otimes \L_\varepsilon))$ (unique up to multiplication by units in $\CC[\varepsilon]/(\varepsilon^2)$) such that $\L_\varepsilon$ is a family of line bundles on $X_r$ and $\omega_\varepsilon$ restricts to $[\omega]$ on the central fiber. We say that $\omega_\varepsilon$ is a deformation of degree $\D$ if $\L_\varepsilon=\pi_1^*(\O_X(\D)\vert_{X_r})$. Of course, the space of degree $\D$ first order deformation of $\F_\omega$ is canonically isomorphic to the tangent space $\T_{[\omega]} \F^1(X,\D)$ to the space of foliations at $[\omega]$. These elements can be expressed as $\omega_\varepsilon=\omega + \varepsilon \eta$ for some $\eta\in  H^0(X,j_*(\Omega^1_{X_r} \otimes \O_{X}(\D)\vert_{X_r})$ such that $\omega_\varepsilon\wedge d\omega_\varepsilon=0$.
 We refer the reader to \cref{section6} for more details on this definition.

A \emph{first order unfolding} of $\F_\omega$ on the other hand corresponds to a twisted differential form $\widetilde{\omega}_\varepsilon\in  H^0(X[\varepsilon], j[\varepsilon]_*(\Omega^1_{X_r[\varepsilon]}\otimes \L_\varepsilon))$ restricting to $[\omega]$ on the central fiber and satisfying  $\widetilde{\omega}_\varepsilon\wedge d\widetilde{\omega}_\varepsilon=0$, where $\L_\varepsilon$ is a family of line bundles on $X_r$. Over a normal base $S$ this would correspond to a foliation on $X\times S$ extending $\F_\omega$. Just as in the previous paragraph, every first order unfolding of degree $\D$ admits a description of the form $\widetilde{\omega}_\varepsilon = \omega+\varepsilon\eta+hd\varepsilon$ where $\eta\in H^0(X,j_*(\Omega^1_{X_r} \otimes \O_{X}(\D)\vert_{X_r}))$ and $h\in H^0(X,\O_X(\D))$. Observe that every degree $\D$ first order unfolding induces a degree $\D$ first order deformation via the natural projection. Along the literature one can find many works where unfoldings are used in order to get a better understanding of the singularities and the deformations of a given foliation, see for instance \cite{suwa-meromorphic,suwa-multiform,MMQ,moli}.

A toric variety $X$ is an algebraic variety with a Zariski open $T\simeq (\C^*)^q$ such that the natural product on $T$ extends to an action on $X$. Along this work, every toric variety will be  assumed to be a normal orbifold (simplicial). The geometry of these spaces can be encoded in a combinatorial object, namely its \emph{fan}. Moreover, one can use this connection between toric and simplicial geometry in order to construct a geometric quotient $(\C^m\setminus Z)/G\simeq X$ for some reductive group $G$. The morphism inducing this isomorphism will be denoted $\pi_X:\C^m\setminus Z\to X$. The homogeneous coordinate ring $S_X$ of $\C^m\setminus Z$ is therefore equipped with a $Cl(X)=Hom(G,\C^*)$-grading and can be used to describe many geometrical objects on $X$.

Let $\F_\omega$ be a foliation on $X$ and $\varphi: Y\dashrightarrow X$ be a generic rational map from a smooth variety $Y$ to $X$. The pullback $\varphi^*\F_\omega$ of $\F_\omega$ under $\varphi$ is defined as the foliation defined by the differential form $\varphi^*\omega\in H^0(Y,\Omega^1_Y\otimes\O_Y(\varphi^*\D))$. These type of foliations have played a central role in the theory of singular foliations, just as in \cite{CLE,PC}.

\

The main topic of this article is the study of foliations on the projective space $\P^n$ of the form $\varphi^*\F$, where $X$ is a toric variety and $\varphi: \P^n\dashrightarrow X$ is a dominant rational map. Pullback foliations admit two canonical types of deformations, namely the ones of the form $\varphi^*\F_\varepsilon$ or $\varphi_\varepsilon^* \F$ where $\F_\varepsilon$ and $\varphi_\varepsilon$ are first order deformations of $\F$ and $\varphi$ respectively. We characterize these deformations in terms of the first order unfoldings of $\varphi^*\F$ and the fibers of $\varphi$ under some generic conditions on the pair $(\varphi,\alpha)$ which are properly stated in \cref{generic}.
In order to do so, we prove some results of independent interest: we give a proper description of both the set of rational maps $\P^n\dashrightarrow X$ and the singularities of a generic foliation $\F$ on $X$ and its pullback $\varphi^*\F$.

\

The content of this work is organized as follows:

In Section 2 we discuss the notions on toric varieties and foliations that will be necessary in the rest of the article. Regarding rational maps $\phi: \P^n\dashrightarrow X$, we show that every such application satisfying $\codim(\varphi^{-1}(Sing(X)))\geq 2$ admits a polynomial lifting in homogeneous coordinates making the follow diagram commute
\[
\xymatrix{
	\CC^{n+1}-\{0\} \ar[d]^{\pi_{\P^n}} \ar@{-->}[r]^F & \CC^m-Z\ar[d]^{\pi_X}\\
	\P^n \ar@{-->}[r]^{\varphi} &X}
\]
and having the \emph{right} base locus, \emph{i.e.}, such that the maximal Zariski open where $\varphi$ is defined is exactly $ \P^n\setminus \pi_{\P^n}\left(F^{-1}(Z)\right)$.

Section 3 is advocated to the study of the singularities of a generic foliation on $X$ and its pullback under a generic rational map. A Kupka singularity of $\F_\omega$ is a point $p\in X$ such that $\omega(p)=0$ and $d\omega(p)\neq 0$. Around these points the structure of $\F_\omega$ is well understood. First, we analyze the singularities of a foliation on a toric surface and determine which are the elements $\D\in Cl(X)$ such that a generic $[\alpha]\in \F^1(X,\D)$ has only reduced Kupka singularities. Then we describe the singularities (both Kupka and non-Kupka) of a generic pullback $\varphi^*\alpha$ in terms of $\varphi^{-1}(Sing(\alpha))$ and different loci associated to the map $\varphi$. For  more
details see  \cref{singcomponents},  \cref{kupkascheme} and  \cref{conclusionaffinegeneric}.

In Section 4 we discuss first order deformations and unfoldings and prove the facts that were mentioned above.
For a rational map $\varphi:\PP^n\dashrightarrow X$ we will denote by $X_r^*$ the intersection of the smooth locus of $X$ and the non-critical values of $\varphi$. The results regarding the canonical types of deformations are the following, see \cref{teo3} and \cref{pullback}:

\begin{theorema}
Let $Y$ be a smooth projective variety with discrete Picard group,  $X$ be a normal projective variety of dimension $q\geq 2$, $\D\in Cl(X)$ and $\varphi: Y \dashrightarrow X$  a dominant rational map with connected fibers such that $\codim (X_r\setminus X_r^*)\geq 2$. Let $\alpha\in \F^1(X,\D)$,  $\omega = \varphi^*(\alpha)$   and let $\omega_\varepsilon$ be a first order deformation of $\omega$. The following conditions are equivalent:
\begin{enumerate}[1)]
	\item The fibers of the map $\overline{\varphi}= \varphi\otimes 1:Y[\varepsilon]\dashrightarrow X[\varepsilon]$ are tangent to $\ker(\omega_\varepsilon)$, \emph{i.e.}, at a generic point $p$ the relative tangent sheaf satisfies
$$\T_{\overline{\varphi}} Y[\varepsilon]_p\subseteq \ker(\omega_\varepsilon)_p\subseteq \T_\NN Y[\varepsilon]_p.$$
	\item There exists a degree $\D$ first order deformation $\alpha_\varepsilon$ of $\alpha$ such that $\omega_\varepsilon=\overline{\varphi}^*(\alpha_\varepsilon)$.
\end{enumerate}
\end{theorema}

\begin{theoremb}
Let $X$ be a normal complete toric surface, $\D\in Cl(X)$, $[\alpha]\in \F^1(X,\D)$, $\varphi:\P^n\dashrightarrow X$ be a rational map, such that $\alpha$ and $\varphi$ are \emph{generic}. Let $\omega_\varepsilon$  be a first order deformation of $\omega:=\varphi^*\alpha$. Then $\omega_\varepsilon$ is induced by a first order unfolding if and only if there is a map $\varphi_\varepsilon:\PP^n[\varepsilon]\dashrightarrow X$ such that
\[
 \omega_\varepsilon=\varphi_\varepsilon^\ast(\alpha).
\]
\end{theoremb}

\

\begin{acknowledgement}
We would like to thank Mariano Chehebar, Fernando Cukierman, Alicia Dickenstein, Alcides Lins Neto and Jorge Vitorio Pereira for fruitful conversations at various stages of this work.  We are especially grateful to Federico Quallbrunn for his further valuable help. Gratitude is also due to Jaroslaw Buczy\'nski for his suggestions and contributions regarding \cref{firstsection}. Finally, we want to thank the anonymous referee for his/her constructive comments on the original manuscript.
\end{acknowledgement}

\section{Toric varieties, foliations and rational maps} \label{firstsection}

In this section we give definitions and basic properties of our main objects of study, namely \emph{toric varieties}, \emph{rational maps $\PP^n\dashrightarrow X$} and \emph{foliations}.

\subsection{Toric varieties and homogeneous coordinates}

\

In the following we will review some facts regarding toric varieties that will be used in the rest of the article. For a complete treatment of this topic and complete proofs we refer the reader to \cite{CLS}. We will follow the notation in \emph{loc. cit.} as well.

\

Let $X=X_\Sigma$ be the toric variety of dimension $q$ associated to the fan $\Sigma$ in $N\simeq \ZZ^q$. We will denote by $\Sigma(1)=\{ \sigma_1,\dots,\sigma_m\} $ the set of one-dimensional cones in $\Sigma$. For each $\sigma_i\in \Sigma(1)$ there is a unique element $v_i\in N$ generating $\sigma_i$.  Unless stated otherwise, we will assume that $X$ is complete and simplicial. In terms of its cones, this is equivalent to requiring $\Sigma$ to cover $\RR^q=N\otimes_\ZZ \RR$ and that the set of rays of every cone in $\Sigma$ can be extended to a basis of $\RR^q$. We will now recall the main insights of the construction of homogeneous coordinates for $X$. 

The action of the open torus $T=N\otimes_\ZZ \CC^*$ on $X$ induces an action on the group $Div(X)$ of divisors of $X$. The subgroup $Div_T(X)$ of elements which are fixed by this action is freely generated by some Weil divisors $D_1,\dots, D_m$ which are in correspondence with the elements in $\Sigma(1)$. Let $M=Hom_\ZZ (T,\CC^*)\subseteq K(X)$ be the character lattice of $T$. The morphism $M\to Div_T(X)$ sending $m\mapsto div(m)$ together with the restriction of the quotient map into the class group $Div_T(X)\to Cl(X)$ fit together in an exact sequence
\begin{equation} \label{suc} 0\rightarrow M \rightarrow Div_T(X)\rightarrow Cl(X) \rightarrow 0.
\end{equation}
Applying the functor $Hom_\ZZ(-,\CC^*)$, we get a presentation of $T$ as a quotient 
$$(\CC^*)^m/G\simeq T,$$
where $G=Hom_\ZZ(Cl(X),\CC^*)$. 
\begin{remark} Let us now fix an isomorphism $Cl(X)\simeq \ZZ^{m-q}\times H$ for a finite group $H$. By duality, this induces an isomorphism $G\simeq (\CC^*)^{m-q}\times Hom_\ZZ(H,\CC^*)$.
\end{remark}
If $v_k=(v_1^k,\dots,v_q^k)$ are the coordinates of the $k$-th ray of $\Sigma$, then the inclusion morphism $G\hookrightarrow (\CC^*)^m$ is given by the equations 
\[
G=\{ (g_1,\dots,g_m)\in (\CC^*)^m \vert \prod_{k=1}^m g_k^{v_i^k}=1 \hspace{0.2cm} \forall 1\leq i \leq q\}\ ,
\]
see for instance  \cite[Lemma 5.1.1, p.~206]{CLS}. It is worth mentioning that this implies that an element $t=(t^{a_1},\dots,t^{a_m})\in (\CC^*)^m$ lies in $G$ if and only if $a_1v_1+\cdots+a_mv_m=0$. The construction of homogeneous coordinates consists of a geometric quotient
\[
\pi_X : \left(\CC^{m}\setminus Z\right)/G\xrightarrow{\sim} X,
\]
where $Z$ is a union of subspaces of codimension equal or greater than 2 which extends the previous presentation. 
The action of $G$ in $S:=\CC[\CC^m\setminus Z]=\CC[z_1,\dots,z_m]$ induces a $Cl(X)=Hom_\ZZ(G,\CC^*)$-grading
\[ S=\bigoplus_{\D\in Cl(X)}S_\D .\]
This graded ring is called the \emph{homogeneous coordinate ring of } $X$. Just as in the projective case, this grading satisfies
\[ H^0(X,\O_X(\D))=S_\D. \]
for every $\D\in Cl(X)$. The coordinate functions are of course homogeneous and satisfy $\deg(z_i)=[D_i]$.
The \emph{irrelevant ideal} of $X$ is the ideal 
\[
I_{Z} = \langle \widehat{z}_{\sigma} = \prod_{i\notin \sigma(1)}z_i \rangle_{\sigma \in \Sigma}.
\]
\begin{remark} \label{Iij} If $X$ is a surface, then either $m=3$ and therefore $Z=\{0\}$ ($X$ is a weighted projective plane) or $Z$ is equidimensional of codimension $2$. In the latter case, if we index the rays in a clock-wise manner we have
\[ Z= \bigcup_{\substack{i,j=1\\i,j \text{ non-consecutive} }}^m Z_{ij} , \]
where  $Z_{ij} = \{z_i=z_j=0\} $. For simplicity, we will use the notation $I_{ij}=I(Z_{ij})$ in  \cref{singularities} and \cref{section6}.
\end{remark}

\begin{remark}  Just as in the projective case, every homogeneous ideal $I\subseteq S$ defines a closed subscheme $\V(I)\subseteq X$.  Moreover, every closed subscheme arises this way (see \cite[Proposition 5.2.7, p.~222]{CLS}]).
\end{remark}

Let $j:X_r\hookrightarrow X$ be the inclusion of the smooth locus of $X$. Recall that the sheaf of Zariski differential $p$-forms is  
\[ \widehat{\Omega}^p_X= (\Omega^p_X)^{\vee\vee}= j_*(\Omega^p_{X_r}).   \] 
The action of $G$ induces $Cl(X)$-grading in the algebra of polynomial differential forms $\Omega^\bullet_S$ (resp. the algebra of polynomial vector fields $\bigwedge^\bullet Der(S,S)$) such that the exterior derivative has degree zero, \emph{i.e.},   $\deg(dz_i)=[D_i]$ (resp. $\deg(\frac{\partial}{\partial z_i})=-[D_i]$).

We will make use of the \emph{generalized Euler sequence}
\begin{equation} \label{1}
0 \longrightarrow \widehat{\Omega}_{X}^{1} \longrightarrow \bigoplus_{i=1}^{m} \O_X(-D_i) \longrightarrow Cl(X)\otimes_\ZZ\O_X  \longrightarrow 0 \ 
\end{equation}
and its dual 
\begin{equation}\label{2}
0 \longrightarrow \O_X^{\oplus m-q} \longrightarrow \bigoplus_{i=1}^m\O_X(D_i) \longrightarrow \T X \longrightarrow 0 ,\
\end{equation}
where $\T X$ is the Zariski tangent sheaf of $X$ (see \cite{CB}). The first arrow in the second sequence corresponds to the \emph{radial vector fields}. These homogeneous polynomial vector fields (of degree zero) generate the tangent space to every $G$-orbit in $\CC^m$ and are of the form
\[ R_j=\sum_{i=1}^{m} a^j_iz_i\frac{\partial}{\partial z_i}   \hspace{0.5cm} 1\le j\le m-q, \]
where the coefficients $a^j_i\in \ZZ$ satisfy $\deg(z_i)=(a_i^1,\dots,a_i^{m-q},h_i) \in Cl(X)$.

The smooth locus $X_r\hookrightarrow X$ corresponds to the subfan $\Sigma_r$ consisting of the smooth cones in $\Sigma$. In particular, we have $\Sigma(1)=\Sigma_r(1)$ and therefore by \eqref{suc} the map
\begin{align}
Cl(X) &\rightarrow Cl(X_r) = Pic(X_r) \\
\notag [\D] &\mapsto [\D|_{X_r}]
\end{align}
is an isomorphism. The construction of homogeneous coordinates for $X_r$ coincides with the (co)restriction of $\pi_X$, \emph{i.e.},
\[X_r\simeq (\CC^m\setminus  Z')/G ,\]
where $Z'=Z\cup \pi_X^{-1}(X\setminus X_r)$. At the level of coordinates rings, 
$$j^*: S_{X_r}=\CC[z_1,\dots,z_m]\to S_X=\CC[z_1,\dots,z_m]$$ 
is the identity. 

\begin{remark}\label{formashomogeneas}
Let $\D$ be a Weil divisor on $X$ and $\L= \O_X(\D)|_{X_r}$ be the line bundle given by the restriction of the reflexive sheaf $\O_X(\D)$ to $X_r$.
If we tensorize the Euler sequence and take cohomology we can conclude that every $\omega \in H^0(X_r, \Omega^1_{X_r}\otimes \L)$ can be described in homogeneous coordinates as a polynomial differential form
 \begin{equation}\label{omega2}
\omega = \sum_{i=1}^m A_i(z)dz_i
\end{equation}
where $A_i\in S_X$ is homogeneous of degree $\L-[D_i]\in Pic(X_r)$ and $\imath_{R_j}(\sum_{i=1}^m A_i(z)dz_i)=0$ for every $1\le j\le m-q$.

 Now, being $\Omega^1_{X_r}\otimes \L$ locally free we have 
\[
 j_*(\Omega^1_{X_r}\otimes \L) \simeq (\Omega^1_X \otimes \O_X(\D))^{\vee \vee},
 \]
see \cite[Proposition 8.0.1, p.~347]{CLS} or \cite[Cor. 1.7, p.~127]{hartstable} as well. This implies
 \[
 H^0(X_r, \Omega^1_{X_r}\otimes \L) \simeq H^0(X, (\Omega^1_{X}\otimes \O_X(\D))^{\vee \vee})\ .
 \]
As a consequence, every $\omega \in H^0(X, (\Omega^1_{X}\otimes \O_X(\D))^{\vee \vee})$ admits a description  like \eqref{omega2} where $\deg(A_i) = [\D-\D_i]\in Cl(X)$ and $\imath_{R_j}(\sum_{i=1}^m A_i(z)dz_i)=0$ for every $1\le j\le m-q$.
\end{remark}

\begin{definition}
We will use the notation
\[
\widehat{\Omega}^{1}_S=\bigoplus_{\D\in Cl(X)} \widehat{\Omega}^{1}_S(\D) =\bigoplus_{\D\in Cl(X)} H^0(X,(\Omega^1_X\otimes \O_X(\D))^{\vee\vee}) \subseteq \Omega^1_S
\]
for the $Cl(X)$-graded $S$-module whose associated sheaf is $\widehat{\Omega}_X^{1}$, see \cite[Corollary 8.1.5, p.~362]{CLS}.
\end{definition}

\begin{remark}\label{camposhomogeneos}
 For vector fields on the other hand, we can repeat the process of the previous remark on the sequence \eqref{2} applied to $X_r$ in order to get
 \begin{equation}\label{pD}\xymatrix @C=20pt @R=10pt{
0 \ar[r] & H^0(X_r,\L^{\oplus m-q}) \ar[r] & H^0(X_r,\bigoplus_{i=1}^m(\L+\L_i)) \ar[r]^-{\rho}  & \\
\ar[r]^-{\rho} & H^0(X_r,\T X_r \otimes \L)  \ar[r] & H^1(X_r,\L^{\oplus m-q}) \ar[r] & \cdots \ .}
\end{equation}
By the arguments above the elements in the image of $\rho$ are exactly the sections
\[\X \in H^0(X_r, \T X_r \otimes \L) \simeq H^0(X, (\T X \otimes \O_X(\D))^{\vee\vee})\]
that admit a homogeneous lifting  to $\CC^m$, \emph{i.e.} a description of the form
\[\X = \sum_{j=1}^m B_j(z) \frac{\partial}{\partial z_j} \]
where $B_j\in \CC[z_1,\dots z_m]$ is homogeneous of degree $[\D+\D_j]\in Cl(X)$. This description is unique up
to linear combinations of the radial vector fields.
\end{remark}

\begin{remark}\label{canonical} The canonical sheaf of a toric variety $X$ is given by $\omega_X \simeq \O_X \left(- \sum_{i=1}^m \D_i\right)$, see  \cite[Theorem 8.2.3, p.~366]{CLS}. This isomorphism can be explicited via the section
\[ \Omega_X:=\ii_{R_1}\dots\ii_{R_{m-q}} dz_1\wedge\dots\wedge dz_m \in H^0(X,(\Omega^q_X\otimes \O_X(\sum_{i=1}^m \D_i))^{\vee\vee}).\]
\end{remark}

\subsection{Rational maps $\PP^n\dashrightarrow X$}
Since we are interested in studying pullbacks of foliations under dominant rational maps $\PP^n\dashrightarrow X$, we will now present a description of these maps that will make our calculations more feasible. We will denote by 
$$\pi:\CC^{n+1}\setminus \{ 0 \} \to \PP^n$$ the standard quotient.

\begin{lemma}\label{homogeneousrationalmaps}
Let $e_1v_1+\dots+e_mv_m=0$ be an equation with integer coefficients for the rays $\{v_i\}_{i=1}^m$ in the fan of $X$. Then, every $F=(F_1,\dots,F_m)\in \CC(x_0,\dots,x_n)^m$, with $F_i$ homogeneous of degree $ e_i$, induces a rational map $\tilde{F}:\P^n\dashrightarrow X$ such that the diagram
\[
\xymatrix{
	\CC^{n+1}-\{0\} \ar[d]^\pi \ar@{-->}[r]^F & \CC^m-Z\ar[d]^{\pi_X}\\
	\P^n \ar@{-->}[r]^{\tilde{F}} &X}
\]
is commutative.
\end{lemma}

\begin{proof} For every $t\in\CC^*$ and every $x\in\CC^{n+1}$ we have 
\begin{align*} 
F(t\cdot x)&=\left(t^{e_1}F_1(x),\dots,t^{e_m}F_m(x)\right)\\
&=(t^{e_1},\dots,t^{e_m})\cdot F(x).
\end{align*} 
The hypothesis on the degrees assures that the point $g(t)=(t^{e_1},\dots,t^{e_m})\in (\CC^*)^m$ satisfies the equations 
\[ \prod_{k=1}^m g(t)_k^{v_i^k}=1 \hspace{0.2cm} \forall 1\leq i \leq q \]
and therefore is an element of $G$. Therefore, the map $F:\CC^{n+1}\setminus 0 \dashrightarrow \CC^m\setminus Z$ descends to a rational map $\tilde{F}:\P^n\dashrightarrow X$ with base locus $\pi(F^{-1}(Z))$.
\end{proof}

Under suitable assumptions every rational map admits a lifting in homogeneous coordinates $F:\CC^{n+1}\setminus \{ 0\} \to \CC^m\setminus Z$.

\begin{lemma}\label{morf1}
Assume  $\Sigma$ has a smooth cone of maximal dimension.
Then for every rational map $\varphi:\P^n\dashrightarrow X$ and every relation $e_1v_1+\dots+e_mv_m=0$ there exist $F=(F_1,\dots,F_m)$ such that $\deg(F_i)= \lambda e_i$, for some $\lambda\in\N$, and $\varphi=\tilde{F}$.
\end{lemma}
\begin{proof} Let $\sigma\in\Sigma$ be a smooth cone with $\dim(\sigma)=q$. Without loss of generality, we can assume that $\sigma(1)=\{v_1, \dots,v_q \}$. Recall we have a smooth  open set $U_\sigma\subseteq X$ with $U_\sigma\simeq \CC^q$. Moreover, by \cite[Proposition 5.2.10, p.~223]{CLS} the restriction of $\pi_X$ to the variety $V=\{x_{q+1}=\cdots=x_m=1\}$ induces an isomorphism with $U_\sigma$. Via this isomorphism, we consider $\varphi: \P^n\dashrightarrow V$.
This map can be described as 
\[
\varphi=\left( \frac{h_1}{g_1},\dots,\frac{h_q}{g_q},1,\dots,1\right),
\]
where $h_i$ and $g_i$ are homogeneous polynomials of the same degree. If we write $g=\prod g_i$, then for a generic point $x\in\CC^{n+1}$ we get $(g^{e_1},\dots,g^{e_m})(x)\in G$. Finally
\begin{align*} \varphi=\left(g^{e_1},\dots,g^{e_m}\right)\cdot \left( \frac{h_1}{g_1},\dots,\frac{h_q}{g_q},1,\dots,1\right)  =\left( g^{e_1}\frac{h_1}{g_1},\dots,g^{e_q}\frac{h_q}{g_q},g^{e_{q+1}},\dots,g^{e_m}\right).  
\end{align*}
Clearly, the lifting $F=(g^{e_1}\frac{h_1}{g_1},\dots,g^{e_q}\frac{h_q}{g_q},g^{e_{q+1}},\dots,g^{e_m})$ satisfies $\deg(F_i)=\deg(g)\,e_i$ and $\varphi=\tilde{F}$.
\end{proof}

\begin{remark} As in the case of projective spaces, the lifting $F$ fails to be unique. Moreover, different liftings give rise to different base loci.
\end{remark}
The problem of deciding whether a rational map between toric varieties admits a \emph{complete} lifting to their respective homogeneous coordinate rings was addressed in \cite{BB}. Here the word `complete' indicates that the lifting has the right base locus. That is, 
\[
Reg(\varphi) = \P^n\setminus \pi\left(F^{-1}(Z)\right),
\]
where $Reg(\varphi)$ is the maximal Zariski open subset on which $\varphi$ is well defined as a regular map. In order to guarantee the existence of such liftings it is necessary to introduce multi-valued maps. Each coordinate of such map is a multi-valued section according to the following.

\begin{definition} A \emph{multi-valued section on $X$} is an element $\gamma$ in the algebraic closure of $\CC(z_1,\dots,z_m)$. We say $\gamma$ is \emph{homogeneous} if $\gamma^r=f\in \CC(z_1,\dots,z_m)$ for some integer $r\geq 1$. In this case, we say that $\gamma$ is \emph{regular} at $p\in \CC^m$ if $f$ is regular at $p$.
\end{definition}

For our purposes it will be highly convenient to restrict ourselves to the set of rational maps admitting a complete polynomial lifting. The following proposition shows that this is the case for a general map $\PP^n\dashrightarrow X$.

\begin{proposition} \label{completelifting} Let $X$ be a complete toric variety and $\varphi:\P^n\dashrightarrow X$ be a dominant rational map satisfying $\codim(\varphi^{-1}(Sing(X)))\geq 2$.
Then $\varphi$ admits a complete polynomial lifting.
\end{proposition}
\begin{proof} 
By \cite[Theorem 4.19, p.~255]{BB} we know that a multi-valued complete lifting $F=(F_1,\dots,F_m)$ exists. We will now prove that if $\codim(\varphi^{-1}(Sing(X)))\geq 2$ then $F$ is actually polynomial. Since $X$ is complete and $\codim(\varphi^{-1}(Sing(X)))\geq 2$ it follows that both $\varphi$ and $\varphi_r:=\varphi|_{\varphi^{-1}(X_r)}$ are regular in codimension 1. Using \cite[Proposition 5.1, p.~259]{BB} we can write
$$ F_i=G_i \gamma_i,$$
where $G_i$ is a rational function and $\gamma_i$ is a multi-valued homogeneous section on $\P^n$, which is also invertible on $Reg(\varphi_r)$. Our hypothesis on $\varphi^{-1}(Sing(X))$ implies that $\gamma_i$ has no zeros or poles in $\P^n$ and therefore $\gamma_i\in\CC$.
In case $G_i$ had poles, $\varphi$ would not be regular in codimension 1. Then it follows that $F_i$ is in fact a polynomial.
\end{proof}

\begin{remark} The hypothesis on $\varphi^{-1}(Sing(X))$ is necessary.
For instance, the map $\P^2\dashrightarrow \P(1,1,2)$ defined  by $F=(z_0^2, z_0 z_1, z_0 z_2^3)$ contracts the divisor $\{z_0=0\}$ into the singular point $[0:0:1]$ and does not admit a complete polynomial lifting.
\end{remark}

\begin{remark} If $F$ is a complete polynomial lifting and $g\in G$, then $g\cdot F$ is again a complete polynomial lifting. In fact, if $\varphi$ is regular in codimension 1 then the complete polynomial lifting is unique up to multiplication by elements of $G$. Its corresponding degree $\bar{e}$ can be computed by looking at the line bundles $\varphi^*(\O_X(D_i))$. Indeed, each coordinate of $F$ satisfies $F_i \in H^0(\P^n,\varphi^*(\O_X(D_i)))$.
\end{remark}

From now on we will only consider rational maps admitting complete polynomial liftings. If no confusion arises, we will use the same notation $F$ for both the rational map and its complete polynomial lifting.

We end this section by showing that if $n\geq m$, then for a generic map $F: \PP^n\dashrightarrow X$ the associated morphism of Cox rings $F^*=S_X\to S_{\PP^n}$ is flat. This condition will be used later in order to deal with generic pullbacks of foliations.

\begin{proposition}Let  $F_1,\dots,F_m$ be homogeneous elements defining a regular sequence in $\CC[x_0,\dots,x_n]$. If $n\geq m$, then the morphism $F^*:\CC[z_1,\dots,z_m]\to \CC[x_0,\dots,x_n]$ such that $z_i\mapsto F_i$ 
is flat.
\end{proposition}
\begin{proof}Let $F_1,\dots,F_m$ be an homogeneous regular sequence in $\CC[x_0,\dots,x_n]$. Since we are assuming $n\geq m$, there must exist some homogeneous polynomials $F_{m+1},\dots,F_{n+1}$ such that the sequence $F_1,\dots F_{n+1}$ is an homogeneous system of parameters. In particular, these elements are algebraically independent and therefore $\CC[F_1,\dots,F_{n+1}]$ is a polynomial ring. Also, by \cite[Theorem 5.9, p.~35]{S} we know that $\CC[x_0,\dots,x_n]$ is a free $\CC[F_1,\dots,F_{n+1}]$-module. Since the map
\begin{align*} \CC[z_1,\dots,z_m] &\to \CC[F_1,\dots,F_{m}]\\
z_i&\mapsto F_i
\end{align*}
is an isomorphism, we can conclude that $\CC[x_0,\dots,x_n]$ is a flat $\CC[z_1,\dots,z_m]$-module and the proposition follows.
\end{proof}

\subsection{Foliations and moduli spaces}

Let $X$ be a  complete normal variety, $X_r$ be the smooth locus of $X$ and $j:X_r \rightarrow X$ its natural inclusion. 

\begin{definition}\label{foliation}
A singular foliation $\F$ of codimension $1$ on $X$ is a singular foliation of codimension $1$ on $X_r$.
\end{definition}

 Being $X_r$ a smooth complex algebraic variety, a singular foliation $\F$ of codimension $1$ on $X_r$ is induced by a twisted differential form $\omega_r \in H^0(X_r,\Omega^1_{X_r}(\L))$, unique up to multiplication by elements in $H^0(X_r,\O_{X_r}^*)$, for some line bundle $\L\in Pic(X_r)$ satisfying the integrability condition
 $\omega_r \wedge d\omega_r = 0$ and vanishing in codimension at least $2$.

Following \cite[Theorem 8.0.4, p.~348]{CLS}, we get $j_*\L = \O_X(\D)$, for some Weil divisor $\D$ on $X$. Then, by \cite[Cor. 1.7, p.~127]{hartstable} and \cite[Theorem 8.0.1, p.~347]{CLS} we have
 \[j_*(\Omega_{X_r}^1\otimes \L) = (\Omega^1_X\otimes \O_X(\D))^{\vee\vee}.\]
Thus, a foliation $\F$ on $X$ gives rise to  a nonzero global section $\omega_\F \in H^0(X,(\Omega^1_X\otimes \O_X(\D))^{\vee\vee})$ such that
\begin{equation}\label{integrability}
\omega_\F \wedge d\omega_\F =0,
\end{equation}
as an element of $H^0(X,(\Omega_X^3\otimes \O_X(\D)^{\otimes 2})^{\vee\vee}) $ and whose zero locus has codimension greater or equal than $2$.

\medskip

 Recall from the previous sections that in the case $X$ is a toric variety every element $\omega \in H^0(X,(\Omega^1_X\otimes \O_X(\D))^{\vee\vee})$ admits a description in homogeneous coordinates of the form
 \begin{equation}\label{omega}
\omega = \sum_{i=1}^m A_i(z)dz_i,
\end{equation}
 where $A_i$ is homogeneous of degree $[\D -\D_i]\in Cl(X)$, satisfying $\imath_{R_k}(\sum_{i=1}^m A_i(z)dz_i)$ for every $1\le k \le m-q$.

 We will use the notation $\omega$ for both the differential form and the foliation induced by it when no confusion arises.

\begin{definition} The singular scheme $Sing(\F_\omega)$ of a foliation $\F_\omega$ is defined as the zero-locus of the global section $\omega\in H^0(X,(\Omega^1_X\otimes \O_X(\D))^{\vee\vee})$. We will denote by $Sing(\F_\omega)_{set}$ its support. In the case $X$ is a toric variety and $\omega$ is as in \cref{omega}, $Sing(\F_\omega)$ is the closed subscheme defined by the homogeneous ideal $J(\omega) = <A_1,\ldots, A_m>$.
\end{definition}

\begin{remark}
Consequently, every foliation of codimension 1 on a toric variety $X$ can be described by a homogeneous differential form $\omega \in \widehat{\Omega}_S^1(\D)$, for some $\D \in Cl(X)$, satisfying in homogeneous coordinates the equation \eqref{integrability} and  $\codim(\V(J(\omega)))\ge2$.

Moreover, two such forms $\omega$ and $\omega'$ define the same foliation if and only if there exists an element $\lambda\in H^0(X,\O_X^*)$ such that $\omega=\lambda\omega'$. We can thus present the spaces parameterizing singular foliations on $X$.
\end{remark}

\begin{definition} Let $X$ be a normal variety and $\D\in Cl(X)$.
The \emph{moduli space of codimension $1$ singular foliations} of degree $\D \in Cl(X)$ on $X$ is the quasi-projective variety
\[
\F^1(X,\D)=\{[\omega]\in \P H^0(X,(\Omega^1_X\otimes \O_X(\D))^{\vee\vee}): \omega \wedge d \omega =0  \,\mbox{ and } \, \codim(Sing(\omega))\geq 2   \}.
\]
\end{definition}

\begin{remark}
When $X$ is a surface, the integrability condition is trivial and therefore the moduli space of degree $\D$ foliations is a Zariski open in  $\P H^0(X,(\Omega^1_X\otimes \O_X(\D))^{\vee\vee})$.

On the other hand, by duality every codimension $1$ foliation on a  surface $X$ is induced by a twisted vector field $Y\in H^0(X,(\T X \otimes \O_X(\D'))^{\vee\vee})$ for some $\D'\in Cl(X)$. In case $X$ is a toric surface, the differential form defining the foliation is
\[\alpha_Y=\ii_Y\Omega_X\in H^0(X,(\Omega^1_X\otimes \O_X(\D'-\omega_X))^{\vee\vee}),\]
where $\ii_Y$ denotes  the contraction with $Y$ and $\Omega_X$ is defined in \cref{canonical}.
\end{remark}

\section{The singular and the Kupka schemes}\label{thirdsection}

We will now define a certain type of singularities that will be central for our analysis, namely Kupka singularities. Throughout this section, unless stated otherwise, we will denote by $X$ a complete simplicial toric variety.

\begin{definition}Let $X$ be a normal variety, $\D\in Cl(X)$ and $\alpha\in H^0(X,$ $(\Omega^1_X\otimes \O_X(\D))^{\vee\vee})$ defining a foliation on $X$. We say that $p\in Sing(\omega)_{set}$ is a \emph{Kupka point} if $d\alpha(p)\neq 0$.
\end{definition}

\begin{remark}
Around a Kupka singularity $p\in X_{r}$ the foliation can be described as the pullback of a germ of foliation on $(\CC^2,0)$ having an isolated singularity at the origin. In particular, the set of smooth Kupka points of a foliation is equidimensional of codimension $2$, see \cite{kupka} or \cite[Fundamental Lemma, p.~406]{medeiros}.
\end{remark}

\begin{definition}\label{kupkadef}
We define the \emph{Kupka set} $\KK_{set}(\alpha)$ as the Zariski closure of the set of Kupka points, that is 
\[
\KK_{set}(\alpha)=\overline{\{p\in Sing(\alpha)_{set}:\  d\alpha(p)\neq 0\}}\ .
\]
\end{definition}

We extend the definitions of \cite[Section 4.2, p.~1035]{MMQ} related to the Kupka singularities of $\alpha$  to this context.

\begin{definition}\label{kupkaideal} The \emph{Kupka scheme} $\KK(\alpha)$ is the schematic support of $d\alpha$ at $\Omega_{S}^2\otimes_S (S\big/J(\alpha))$. Then, $\KK(\alpha)\subseteq X$ is the scheme associated to the homogeneous ideal 
\[
K(\alpha)=\ann(\overline{d\alpha})+J(\alpha)\subseteq S,  \hspace{0.2cm} \mbox{where} \hspace{0.2cm}\overline{d\alpha}\in \Omega_{S}^2\otimes_S (S\big/J(\alpha)).
\]
\end{definition}

Let us recall the notion of
\emph{ideal quotient} of two $S$-modules $M$ and $N$
\[
(N:M) := \left\{a\in S: a.M\subseteq N\right\}.
\]
See \cite[Exercise 1.12, p.~8 and Corollary 3.15, p.~43]{atiyahmacdonald} for basic properties. In the case of two ideals $I,J\subseteq S$ the \emph{saturation} of $J$ with respect to $I$ is defined as
\[
\left(J:I^\infty\right) := \bigcup_{d\geq 1} \big(J:I^d\big).
\]
Then it is also possible to describe $K(\alpha)$ as 
\begin{equation*}\label{Kbisbis}
K(\alpha)=(J\cdot \Omega_S^2: d\alpha)\ . 
\end{equation*}
Let us denote  the ideal of polynomial coefficients of a differential form $\eta\in\Omega_S^r$ by $\II(\eta)$. With this notation, we have $J(\alpha)=\II(\alpha)$. Since $\Omega_S^2$ is a free module, we also get
\begin{equation}\label{Kbis}
K(\alpha)=(J(\alpha):\II(d\alpha))\ .
\end{equation}
As for the Kupka set, its homogeneous ideal coincides with the saturation of $J(\alpha)$ with respect to $\II(d\alpha)$, \emph{i.e.},
\begin{equation}\label{Ksetbis}
K_{set}(\alpha) = \sqrt{(J(\alpha):\II(d\alpha)^\infty)}. 
\end{equation}

\begin{remark}
The reduced structure of the Kupka scheme may be supported on a bigger space than the Kupka set. See for instance \cite[Example 4.5, p.~1034]{MMQ}. However, when the singular locus is reduced we have the following.
\end{remark}

\begin{lemma}\label{K=Kset} If $p$ is a reduced point of $Sing(\alpha)$,  \emph{i.e.} $J(\alpha)_p=\sqrt{J(\alpha)}_p$, then
$$
K(\alpha)_p = K_{set}(\alpha)_p.
$$
\end{lemma}
\begin{proof}
This follows immediately from the equalities
\[
K(\alpha)_p=(J(\alpha)_p:\II(d\alpha)_p)=(J(\alpha)_p:\II(d\alpha)^\infty_p)=K_{set}(\alpha)_p.
\]
\end{proof}

\subsection{Singularities of foliations in toric surfaces}
\label{singularities}
The aim of this section is to describe the singular scheme of a generic foliation on a toric surface $X$.
Recall that every foliation on $X$ is induced by a twisted vector field $Y\in H^0(X,H^0(X, (TX\otimes \O_X(\D))^{\vee\vee})$ for some $\D\in Cl(X)$. In this case, the Zariski differential form defining the foliation is $\alpha_Y=\ii_Y\Omega_X\in H^0(X,(\Omega^1_X\otimes \O_X(\D-\omega_X))^{\vee\vee})$.

We will begin by analyzing the singularities that are contained in $X_{r}$. A similar approach has also been used in \cite[Section 2.3, p.~6]{CP2}.

\begin{proposition}\label{kupkaestodo} Let $X$ be a toric surface, $\D\in Pic(X)$ a very ample line bundle and $\D'\in Cl(X)$ such that the natural map 
$$H^0(X,(T X\otimes \O_X(\D'))^{\vee\vee})\otimes_\C \O_X \to (T X \otimes \O_X(\D'))^{\vee\vee}$$ is non-zero at every fiber over $X_r$. Then for a generic $Y\in H^0(X, (TX \otimes \O_X(\D+\D'))^{\vee\vee})$, the scheme $Sing(\alpha_Y)\cap X_{r}$ is reduced and consists only of Kupka singularities.
\end{proposition}
\begin{proof}Let $d=\dim  (\P H^0(X,(TX \otimes \O_X(\D+\D'))^{\vee\vee}))$ and 
$$\W\subseteq \P H^0(X,(TX \otimes \O_X(\D+\D'))^{\vee\vee})\times X_{r}$$ 
be the incidence variety defined by
\begin{equation*}
\W=\{ (Y,p) \ : \,   \mbox{ $\alpha_Y(p)= d(\alpha_Y)(p)=0$ or $p$ is a non-reduced point of } Sing(\alpha_Y) \}
\end{equation*}
and equipped with both projections  $\xymatrix@1{   \P H^0(X,(TX \otimes \O_X(\D+\D'))^{\vee\vee})& \W \ar[r]^-{\pi_2} \ar[l]_-{\pi_1} & X_r}$. For every $p\in X$ the space of sections $Z$, where $Z(p)=0$, identifies with the kernel of the evaluation $ev_p$ of vector fields at $p$. Therefore, it determines a codimension $2$ linear space in $ \P H^0(X,(TX \otimes \O_X(\D+\D'))^{\vee\vee})$. 
 Our hypotheses on $\D$ and $\D'$ are in order to guarantee that for each $p\in X_r$ there exists an element $Y\in H^0(X,(TX \otimes \O_X(\D+\D'))^{\vee\vee})$ such that $p$ is a reduced Kupka singularity of the foliation induced by $Y$. Indeed, if $Z$ is a global section of $(TX\otimes\O_X(\D'))^{\vee\vee}$ not vanishing at a point $p$ and $f\in \O_X(\D)$ such that $\{ f=0\}$ is reduced and not tangent to $Y$ at $p$, then $Y=f Z$ satisfies the above.
This implies that the fiber $\pi_2^{-1}(p)$ has dimension  $\le d-3$. By the Fiber Dimension Theorem \cite[Theorem 3, p.~49]{M}
\[
\dim(\W)\leq d-3+2=d-1.
\]
In particular, a generic  $Y\in \P H^0(X,(TX \otimes \O_X(\D+\D'))^{\vee\vee})$ is not in the image of $\pi_1(\W)$ and hence the claim. 
\end{proof}

\begin{remark} The proposition above shows that under the right hypotheses on $\D$ the non-Kupka points of a generic foliation of degree $\D$ are contained in the singular points of $X$. In order to assure that every singularity of a generic foliation is a Kupka singularity, it would be sufficient to assure that $Sing(\F)\cap Sing(X)=\emptyset$. This is an \emph{arithmetical} condition on $\D$: it is equivalent to requiring that not every homogeneous element in $\widehat{\Omega}^1_S$ of degree $\D$ vanishes on each of the orbits in $\pi_X^{-1}(Sing(X))$.
\end{remark}

\begin{example}
Let $X=\P^2(a_0,a_1,a_2)$. We will always assume that the weights $a_i$ are coprime by pairs.
Let $\alpha =\sum_{i=0}^2 A_i(z)dz_i \in \widehat{\Omega}^1_S(d)$ and $Z$ an homogeneous polynomial vector field. We recall the Cartan formula
\[ \imath_Zd{\alpha} + d\imath_Z{\alpha} = L_Z({\alpha}) \]
where $L_Z(\alpha)$ denotes the Lie derivative of $\alpha$ with respect to  the vector field $Z$. Using this equation in the case $Z=R = \sum_{i=0}^2 a_iz_i \tfrac{\partial}{\partial z_i}$ 
we get 
\begin{equation}\label{euler3}
\imath_Rd{\alpha} + d\imath_R{\alpha} = deg({\alpha}){\alpha}\ .
\end{equation}
Following \cite[Theorem (ii) p.~39]{DOLGACHEV} every twisted vector field in $\P^2(a_0,a_1,a_2)$ admits a description as an homogeneous polynomial vector field.
For an element
\[Y=\sum_{i=0}^2 B_i \frac{\partial}{\partial z_i} \] 
 of degree $\ell$, let $\alpha_Y=\ii_Y\Omega_X\in \widehat{\Omega}^1_S(\ell+a_0+a_1+a_2)$ and $div(Y)= \sum_{i=0}^2\frac{\partial B_i}{\partial z_i}$ be its divergence. 
Using  \cref{euler3} straightforward computation shows
\begin{equation}\label{formula}
d{\alpha_Y} = div(Y)\Omega_{X} - (\ell+a_0+a_1+a_2)\, \imath_Y dz_0\wedge dz_1\wedge dz_2,
\end{equation}
The vector $Y - \frac{div(Y)}{\ell+a_0+a_1+a_2}R$ defines the same element in $H^0(X, (\T X\otimes\O_X(\ell))^{\vee\vee})$ as $Y$ and satisfies $div(Y)=0$. Thus we can assume that the foliation is induced by an homogeneous polynomial vector field
with zero divergence. 
From  \cref{formula} we can deduce 
\[Sing(d\alpha_Y)_{set}= \{z\in \CC^3: B_0(z)=B_1(z)=B_2(z)=0\}.\]
Recall that  the singularities of $\P^2(a_0,a_1,a_2)$ are of the form $p_0=[1:0:0]$, $p_1=[0:1:0]$ or $p_2=[0:0:1]$, depending on whether the corresponding $a_i>1$. The degrees $\ell$ such that a generic foliation satisfies $Sing(\F)\cap Sing(X)=\emptyset$ must then coincide with the $\ell\in \ZZ$ admitting
at least one vector field  $Y = \sum_{i=0}^2 B_i \tfrac{\partial}{\partial z_i}$ of degree $\ell$ with
\[\{p\in\P^2(a_0,a_1,a_2): B_0(p)=B_1(p)=B_2(p)=0\}\cap Sing(\P^2(a_0,a_1,a_2))=\emptyset.\]
Straightforward calculation shows that this are exactly the degrees satisfying 
\[
\ell+a_0\equiv 0 \ \text{mod}(a_i) \text {\ \ \  or  \ \ \  }\ell+a_1\equiv 0 \ \text{mod}(a_i) \text {\ \ \ or\ \ \  }\ell+a_2\equiv 0 \ \text{mod}(a_i)
\]
for every $0\leq i\leq 2$ such that $a_i>1$.
\end{example}

Let us now state the genericity conditions that will be needed in the following sections.

\begin{definition} Let $\D\in Cl(X)$ and $\alpha\in \widehat{\Omega}^1_S(\D)$. We say that $\alpha$ is \emph{generic} if every singularity is a reduced Kupka singularity, \emph{i.e.}, if
$$(J(\alpha):I_Z^\infty)=(K(\alpha):I_Z^\infty)$$ is radical, where $I_Z$ is the irrelevant ideal of $X$.
\end{definition}

\begin{remark}\label{remKupka}
Let $X$ be a toric surface other than $\P^2(a_0,a_1,a_2)$ and $\D\in Cl(X)$. Then for a general $\alpha \in \widehat{\Omega}^1_S(\D)$ , the ideal $J(\alpha)$ could have components supported on primes of $I_Z$ that appear outside the ideal of the Kupka set and even the Kupka scheme. 
\end{remark}

Let $\D\in Cl(X)$ and $\F$ be the foliation on $X$ induced by a  differential form $\alpha=\sum_{i=1}^mA_i(z) dz_i\in \widehat{\Omega}^1_S(\D)$. Its pullback by $\pi_X$ induces a foliation on $\CC^m\setminus Z$ (and therefore on the affine space $\CC^m$) satisfying
\[ \omega_{\pi_X^*\F}=\alpha=\sum_{i=1}^m A_i(z) dz_i .\]

\begin{proposition} \label{equidim}Let $X$ be a toric surface and $\F$ a foliation on $X$. Then $Sing(\pi_X^*\F)$ is equidimensional of codimension $2$. In particular, if the polynomial differential form $d\alpha$ does not vanish in codimension $2$, then $\sqrt{J(\alpha)}=\sqrt{K(\alpha)}$.
\end{proposition}
\begin{proof} Since $\pi_X^*\F$ is a codimension $1$ foliation whose tangent sheaf splits as a sum of line bundles we can proceed as in the proof of \cite[Proposition 12, p.~11]{V}.
\end{proof}

If $\F$ is generic then every singularity $p\in \CC^m\setminus Z$ of $\pi_X^*\F$ is reduced and of Kupka type. It could be the case, however, that other type of singularities appear inside $Z$.
To have control over such a situation,  we fix the following:

\begin{definition}\label{gamaalfa}
When $m=|\Sigma(1)|>3$, we define
\begin{align*}
&\Gamma_\alpha=\{\text{non-consecutive pairs }(i,j):\ i<j,\ J(\alpha)\subset I_{ij}\}\ ,\\
&\Gamma_{\alpha,\K} =\{(i,j)\in \Gamma_{\alpha,\K}:\ K(\alpha) \subset I_{ij} \}\ \text{ and }\\
&\Gamma_{\alpha,\K}^{set}=\{(i,j)\in \Gamma_{\alpha,\K}:\ K_{set}(\alpha) \subset I_{ij} \}\ ,
\end{align*}
where $I_{ij}$ is defined in \cref{Iij}. If $m = 3$, we consider $ \Gamma_\alpha= \Gamma_{\alpha,\K} = \Gamma_{\alpha,\K}^{set}= \emptyset$.
\end{definition}

\begin{definition}  Let $\D\in Cl(X)$ and $\alpha\in \widehat{\Omega}^1_S(\D)$. We say that $\alpha$ is \emph{affinely generic} if $Sing(\pi^*\F)$ is reduced and every irreducible component of $Sing(\pi^*_X\F)$ consists generically of Kupka points, \emph{i.e.}, if $J(\alpha)=K(\alpha)$ is radical.
\end{definition}

\subsection{The singular scheme of a pullback foliation}

Let $X$ be a simplicial complete toric surface, $\alpha \in \widehat{\Omega}^1_S(\D)$ and $F:\P^n\dashrightarrow X$ be a dominant rational map with a complete polynomial lifting of degree $\overline{e}$. We will  write $\omega=F^*\alpha$ and suppose that $n>m$. This hypothesis guarantees that the first condition in the following definition is indeed generic in the corresponding space.

\begin{definition}\label{generic}We say that the pair $(F,\alpha)$ is \emph{generic} if $\alpha$ is generic and
\begin{enumerate}[I)]
	\item the critical points $C(F)$ of the morphism $F:\CC^{n+1}\to \CC^m$ are of codimension at least three and
	\item the scheme $Sing(\omega)$ is reduced along 
\[ C(F,\alpha):=\overline{\{p\in C(F) : \alpha(F(p))\neq 0 \mbox{ and } Im(d_pF)\subseteq \ker(\alpha(F(p)) \}} \]
and has no embedded components supported at $Sing(\omega)\cap C(F)$.
\end{enumerate}
\end{definition}

\begin{definition} \label{affgeneric} A generic pair $(F,\alpha)$ is said to be \emph{affinely generic} if $\alpha$ is affinely generic and the scheme associated to the ideal $F^*K(\alpha)=\langle A_1(F),\dots, A_m(F)\rangle$ has no embedded components supported at the origin $0\in \CC^{n+1}$.
\end{definition}
\begin{remark} The conditions above define a Zariski open set in the space of pairs 
\[\{ (F,\alpha) : \deg(F)=\overline{e} \mbox{ and } \deg(\alpha)=\D \}. \]
We refer to \cite[Th\'eor\`eme (12.2.4), p.~183]{G} for this property on condition $\textit{II)}$. 
\end{remark}

\begin{lemma}\label{singcomponents} The singular set of $\omega$ is
\[ Sing(\omega)_{set}=\overline{ F^{-1}(Sing_{set}(\alpha))} \cup \bigcup_{(i,j)\in\Gamma_\alpha}\{ F_i=F_j=0\} \cup C(F,\alpha), \]
where each of these varieties do not share irreducible components and $C(F,\alpha)$ consists of non-Kupka points.
\end{lemma}
\begin{proof} Recall from \cref{equidim} that the zero locus of $\alpha$ in $\CC^m$ is equidimensional of codimension $2$. In particular, its irreducible components contained in $Z$ are of the form $\{z_i=z_j\}=0$ for $(i,j)\in \Gamma_\alpha$. Then every singular point $p$ of $\omega$ such that $\alpha(F(p))=0$ must lie in the first two elements of the above union.
 
Now let $p\in Sing(\omega)_{set}$ be such that $\alpha(F(p))\neq 0$. Being $p$ a singularity of $\omega$, the image of $d_p{F}$ is contained in $ker (\alpha(F(p)))$, \emph{i.e.}, $p\in C(F,\alpha)$. It follows from the integrability condition $\alpha\wedge d\alpha=0$ that $d\alpha(v,w)({F}(p))=0$ for every $v,w\in ker (\alpha(F(p)))$. Indeed, if $v,w\in \ker(\alpha)$ then
\[ \ii_{v\wedge w}\left(\alpha\wedge d\alpha\right)=\ii_{v\wedge w}(d\alpha)\alpha=0\]
and therefore $d\alpha(v,w)=0$. Since ${F}^*(d\alpha) = d\omega$,  we get  $d\omega(p)=0$. 
\end{proof}
As for the Kupka set, we have
\begin{lemma}\label{kupka} If $(F,\alpha)$ is generic, then 
\[
\KK_{set}(\omega)= \overline{F^{-1}(Sing_{set}(\alpha))} \cup \bigcup_{(i,j)\in \Gamma_{\alpha,\K}^{set}} \{F_i=F_j=0\}.
\]
\end{lemma}

\begin{proof} It follows from the above proof that for every Kupka point $p$ of $\omega$ its image $F(p)$ must satisfy $\alpha(F(p))=0$ and $d\alpha(F(p))\neq 0$. Being $\alpha$ generic, its  Kupka set in $\CC^m$ is the union of $ \pi^{-1}_X(Sing(\alpha)_{set})$ and  $\{z_i=z_j=0\}$ for $(i,j)\in \Gamma_{\alpha,\K}^{set}$. On the other hand, being $(F,\alpha)$ generic we can assure that every element in
$$F^{-1}(Sing_{set}(\alpha)) \cup \bigcup_{(i,j)\in \Gamma_{\alpha,\K}^{set}} \{F_i=F_j=0\}$$ 
is generically of Kupka type and the lemma follows.
\end{proof}

We shall now describe the Kupka and singular schemes of a generic pullback foliation.

\begin{remark}\label{remJmono}  Let $\alpha=\sum_{i=1}^m A_i(z) dz_i$.
Observe that the ideal $J(\omega)$ is generated by 
\[
\sum_{i=1}^{m} A_i(F)\frac{\partial F_i}{\partial x_j},
\] 
for $j=0\dots,n$. Let $J_F$ be the Jacobian matrix of the polynomial lifting  $F:\CC^{n+1}\to \CC^m$. Then, $J(\omega)$ is in fact generated by  the entries of the vector 
\[
J_F^t \cdot (A_1(F),\dots,A_m(F)).
\]
Also, the coefficients of the differential form $d\omega$ are the entries of 
\[ J_F^t \cdot (A_{ij}(F)) \cdot J_F,\]
where $A_{ij}$ satisfies $A_{ij}=-A_{ji}$ and $d\alpha=\sum_{i,j=1}^{m} A_{ij} dz_i\wedge dz_j\in \Omega^2_{S}$.
In particular, for a point $p\in \PP^n$  outside $C(F)$ we have
\[ J(\omega)_p=F^*(J(\alpha))_p \]
and
\[\II(d\omega)_p=F^*(\II(d\alpha))_p .\]
\end{remark}

\begin{theorem}\label{kupkascheme} Let $(F,\alpha)$ be generic. Suppose further that the ring homomorphism $F^*:S_X\to S_{\PP^n}$ is flat. Then 
$\K(\omega)$ is the closed subscheme defined by the ideal $F^*K(\alpha)$.
\end{theorem}
\begin{proof} 
 Being $(F,\alpha)$ generic we know that $Sing(\omega)$ is reduced along $C(F,\alpha)$.  Since these points are not of Kupka type, $\K(\omega)$ has no irreducible component supported at $C(F,\alpha)$ (because by \cref{K=Kset} such irreducible component would actually lie in $\K_{set}(\omega)$).

By conditions I) and II) in \cref{generic}, the Jacobian matrix $J_F$ has maximal rank at a generic point of an irreducible component of $Sing(\omega)$ supported at
\[ \overline{F^{-1}(Sing_{set}(\alpha))} \cup \bigcup_{(i,j)\in \Gamma_\alpha} \{F_i=F_j=0\}. \]

By the above remark, if $p$ is the generic point  of an irreducible component of the Kupka scheme $\K(\omega)$ we have 
\[ K(\omega)_p=( J(\omega)_p: \II(d\omega)_p )=( F^*(J(\alpha))_p: F^*(\II(d\alpha))_p ).  \]
Since $F^*$ is flat, by \cite[Chap. 1, §2, Remarque, p.~41]{B}, we can conclude that $K(\omega)_p=F^*K(\alpha)_p$ and hence the claim.
\end{proof}

\begin{lemma}\label{cuenta} Let $\alpha\in \widehat{\Omega}^1_S(\D)$ of the form $\alpha=\sum_{i=1}^m A_i(z) dz_i $ and $\omega=F^*\alpha$. Then $A_k(F) d\omega= \omega\wedge \eta_k$ where $\eta_k=F^*(\ii_{\frac{\partial}{\partial z_k} }d\alpha)\in \Omega^1_S$.
\end{lemma}
\begin{proof}
 Contracting the equation
$\alpha\wedge d\alpha = 0$ 	with the vector field $\frac{\partial}{\partial z_k}$ we get
$$A_k d\alpha = \alpha\wedge (\ii_{\frac{\partial}{\partial z_k}}d\alpha).$$
Since the pullback commutes with the exterior differential we have
\begin{equation*}\label{ecu11}
A_k(F) d\omega= \omega\wedge \eta_k ,
\end{equation*}
where $\eta_k = F^\ast(\ii_{\frac{\partial}{\partial z_k}}d\alpha) \in \Omega_S^1 $.
\end{proof}

In the case where $(F,\alpha)$ is affinely generic, this easily implies the following.
\begin{corollary}\label{kupkagen} If $(F,\alpha)$ is affinely generic and $F^*:S_X\to S_{\PP^n}$ is flat, then 
\[ K(\omega)=\langle A_1(F),\dots,A_m(F)\rangle. \]
\end{corollary}
\begin{proof} Recall that $K(\omega)=(J(\omega)\cdot \Omega^2_S : d\omega)$. Since $\alpha$ is affinely generic, by \cref{cuenta} above  we have
$$F^*K(\alpha)=\langle A_1(F),\dots,A_m(F)\rangle \subseteq K(\omega).$$
By \cref{kupkascheme} and our hypothesis on $(F,\alpha)$ being affinely generic we know that
\[ (K(\omega): \langle x_0,\dots,x_n\rangle^\infty)=(F^*K(\alpha): \langle x_0,\dots,x_n\rangle^\infty)=F^*K(\alpha). \] 
In particular,  $K(\omega)\subseteq F^*K(\alpha)$  and therefore $K(\omega)=\langle A_1(F),\dots,A_m(F)\rangle$ .
\end{proof}

\begin{corollary}\label{conclusionaffinegeneric} If $(F,\alpha)$ is affinely generic and $F^*:S_X\to S_{\PP^n}$ is flat, then $Sing(\omega)$ is reduced and 
\[ Sing(\omega)=\K_{set}(\omega)\cup C(F,\alpha) .\]
\end{corollary}
\begin{proof} Recall that the support of $Sing(\omega)$ is 
\[ Sing(\omega)_{set}=\overline{ F^{-1}(Sing_{set}(\alpha))} \cup \bigcup_{(i,j)\in\Gamma_\alpha}\{ F_i=F_j=0\} \cup C(F,\alpha). \]
Being $\alpha$ affinely generic we can assure that $\Gamma_\alpha=\Gamma_{\alpha,\K}^{set}$. By \cref{kupka}, the equality
\[\K_{set}= \overline{ F^{-1}(Sing_{set}(\alpha))} \cup \bigcup_{(i,j)\in\Gamma_\alpha}\{ F_i=F_j=0\} \]
also holds. 

By hypothesis $Sing(\omega)$ is reduced along $C(F,\alpha)$. On the other hand, being $\alpha$ affinely generic we know that $J(\alpha)=K(\alpha)$ is radical. Since the morphism $F^*$ is flat, by \cref{kupkascheme} we can also conclude that $Sing(\omega)$ is reduced along $\K_{set}(\omega)$.
\end{proof}

\section{First order unfoldings and deformations}\label{section6}

This section is dedicated to the study of first order deformations and unfoldings of pullback foliations by dominant rational maps $F:\PP^n\dashrightarrow X$.

\

Let $\NN=Spec(\CC[\varepsilon]/(\varepsilon^2))$ be the spectrum of the dual numbers. For a variety $X$ with smooth locus $X_r$ we will use the notation $X[\varepsilon]=X\times \NN$, $\pi_1:X[\varepsilon]\to X$ for the first projection and $j[\varepsilon]: X_r[\varepsilon]\hookrightarrow X[\varepsilon]$ the natural inclusion.

\begin{definition}Let $X$ be a normal variety, $\D\in Cl(X)$ and $\alpha\in H^0(X,j_*(\Omega^1_{X_r} \otimes \O_{X}(\D)\vert_{X_r}))$ a differential form defining a foliation on $X$. A \emph{first order deformation of} $\F_\alpha$ is a family of foliations on $X$ parameterized by $\NN$. This is induced by a unique integrable element $\alpha_\varepsilon\in H^0(X[\varepsilon], j[\varepsilon]_*(\Omega^1_{X_r[\varepsilon]|\NN}\otimes \L_\varepsilon))$ (up to multiplication of global sections of $\O_{X[\varepsilon]}^*$), where $\L_\varepsilon$ is a family of line bundles on $X_r$ and $\alpha_\varepsilon$ restricts to $\alpha$ on the central fiber. We say that $\alpha_\varepsilon$ is of degree $\D$ if $\L_\varepsilon=\pi_1^*(\O_X(\D))$.
\end{definition}

\begin{remark} Observe that a degree $\D$ first order deformation is (up to multiplication by units) an integrable global section of $\pi_1^*(j_*(\Omega^1_{X_r} \otimes \O_X(\D)\vert_{X_r})$. This naturally identifies with the Zariski tangent space $\T_{[\alpha]} \F^1(X,\D)$ to the space of foliations at $[\alpha]$. In particular such a deformation admits a unique description of the form  $\alpha_\varepsilon=\alpha + \varepsilon \eta$ for some $[\eta]\in H^0(X,j_*(\Omega^1_X\otimes\O_X(\D)\vert_{X_r}))/\langle \alpha \rangle$ satisfying the integrability condition
\begin{equation}\label{eqdeformations}
\alpha \wedge d\eta + \eta \wedge d\alpha = 0.
\end{equation}
\end{remark}

Let us now move on to a different type of first order perturbation. Following \cite[2nd. paragraph, p.~11]{MMQ2} for a global version or \cite[(6.9), p.~201]{suwa} for a local one, a first order unfolding of degree $\L$ of a foliation $\F_\alpha$  on a smooth variety $Y$ induced by some $\alpha\in H^0(Y, \Omega^1_Y\otimes \L)$  is given by an element $\widetilde{\alpha}_\varepsilon \in H^0(Y[\varepsilon],\Omega^1_{Y[\varepsilon]}\otimes \L_\varepsilon)$ satisfying $\widetilde{\alpha}_\varepsilon\wedge d\widetilde{\alpha}_\varepsilon=0$, where $\L_\varepsilon$ is a first order deformation of $\L$. In the case $\L_\varepsilon=\pi_1^*\L$,  such an element must be of the form
\[ \widetilde{\alpha}_\varepsilon= \alpha + \varepsilon \eta + h d\varepsilon \]
for some $\eta\in H^0(Y,\Omega^1_Y\otimes \L)$ and $h\in H^0(Y,\L)$. In terms of these sections, the integrability condition  $\widetilde{\alpha}_\varepsilon\wedge d\widetilde{\alpha}_\varepsilon=0$ is guaranteed by the equations
\begin{equation}\label{equnfoldings}
\left\{\begin{aligned}
\alpha\wedge d\eta + d\alpha\wedge\eta = 0\\
h d\alpha = \alpha \wedge (\eta - dh)\ 
\end{aligned}\right.\qquad\iff \qquad h d\alpha = \alpha \wedge (\eta - dh)\ .
\end{equation}
This last equivalence will be detailed in the proof of \cref{unfdef}. Over a normal variety, it is natural to define the following:

\begin{definition} Let $X$ be a normal variety, $\D\in Cl(X)$ and $\F_\alpha$ a foliation induced by some element $\alpha\in H^0(X,j_*(\Omega^1_X\otimes\O_X(\D)\vert_{X_r}))$. A first order unfolding of $\F_\alpha$ is a first order unfolding of $\F_\alpha\vert_{X_r}$, \emph{i.e.}, an integrable element $\widetilde{\alpha}_\varepsilon\in H^0(X[\varepsilon], j[\varepsilon]_*(\Omega^1_{X_r[\varepsilon]}\otimes \L_\varepsilon))$ restricting to $\alpha$ on the central fiber, where $\L_\varepsilon$ is a family of line bundles on $X_r$. We say that $\alpha_\varepsilon$ is of degree $\D$ if $\L_\varepsilon=\pi_1^*\O_X(\D)$.
\end{definition}

Just as in the smooth case, if $\alpha \in H^0(X,j_*(\Omega^1_X\otimes\O_X(\D)\vert_{X_r}))$ is an element defining a foliation $\F_\alpha$ on $X$ then every degree $\D$ first order unfolding of  $\F_\alpha$ admits a description
\[ \widetilde{\alpha}_\varepsilon= \alpha + \varepsilon \eta + h d\varepsilon \]
for some $\eta\in H^0(X,j_*(\Omega^1_X\otimes\O_X(\D)\vert_{X_r}))$ and $h \in H^0(X,\O_X(\D))$ satisfying \cref{equnfoldings}. These elements are of course the natural extensions of the corresponding elements in $X_r$.


\begin{definition}\label{defU}
The space of degree $\D$ first order unfoldings of $\alpha$ is
\[
U(\alpha) = \{(h,\eta)\in H^0(X,\O_{X}(\D)\oplus j_*(\Omega^1_X\otimes\O_X(\D)\vert_{X_r}))  :  h d\alpha = \alpha \wedge (\eta - dh) \}\big / \langle (0,\alpha)\rangle.
\]
\end{definition}

\begin{lemma}\label{unfdef}
Every element $(h,\eta)\in U(\alpha)$ induces a degree $\D$ first order deformation $\alpha_{\varepsilon}= \alpha + \varepsilon \eta$.
Then, we have a well-defined map
\begin{align*}
U(\alpha) &\xrightarrow{\pi_2} \T_{[\alpha]} \F^1(X,\D) \\
(h,\eta) &\longmapsto  [\eta] \ .
\end{align*}
\end{lemma}

\begin{proof}
See \cite[Proposition 2.6]{moli}.
\end{proof}
%
%

Now let $X$ be a toric variety, $F:\P^n\dashrightarrow X$ be a dominant rational map with a complete polynomial lifting of degree $\bar{e}$ and $\alpha$ be a twisted differential form defining a foliation on $X$. Its pullback $\omega:=F^\ast (\alpha)$ admits non-trivial canonical deformations:
if $F_\varepsilon = F + \varepsilon G:\P^n \dashrightarrow X[\varepsilon]$ and $\alpha_\varepsilon = \alpha+\varepsilon\eta $ are first order deformations of $F$ and $\alpha$ respectively, then one can consider the family $\omega_\varepsilon:=F_\varepsilon^*(\alpha_\varepsilon)$. We aim to describe these deformations in a more explicit manner. 

The following theorem is a version of \cite[Lemma 6.7, p.~89]{AD} and  \cite[Lemma 2.2, p.~162]{CLNLPT} for first order deformations. For a dominant rational map $F: Y\dashrightarrow X$ from a smooth variety into a normal (not necessarily toric) $X$ we will denote by $X_r^*$ the intersection of the smooth locus of $X$ and the non-critical values of $F$.

\begin{theorem} \label{teo3}
Let $Y$ be a smooth projective variety of dimension $n$ with discrete Picard group,  $X$ be a normal projective variety of dimension $q\geq 2$, $\D\in Cl(X)$ and $F: Y \dashrightarrow X$  a dominant rational map with connected fibers such that $\codim (X_r\setminus X_r^*)\geq 2$. Let $\alpha\in \F^1(X,\D)$,  $\omega = F^*(\alpha)$   and let $\omega_\varepsilon$ be a first order deformation of $\omega$. The following conditions are equivalent:
\begin{enumerate}[1)]
	\item The fibers of the map $\overline{F}= F\otimes 1:Y[\varepsilon]\dashrightarrow X[\varepsilon]$ are tangent to $\ker(\omega_\varepsilon)$, \emph{i.e.}, at a generic point $p$ the relative tangent sheaf satisfies
$$\T_{\overline{F}} Y[\varepsilon]_p\subseteq \ker(\omega_\varepsilon)_p\subseteq \T_\NN Y[\varepsilon]_p.$$
	\item There exists a degree $\D$ first order deformation $\alpha_\varepsilon$ of $\alpha$ such that $\omega_\varepsilon=\overline{F}^*(\alpha_\varepsilon)$.
\end{enumerate}
\end{theorem}
\begin{proof}  It follows from the definitions that 2) implies 1). 
We shall now prove the other implication. For the sake of clarity we will denote by $p_1:Y[\varepsilon]\to Y$ and $\pi_1:X[\varepsilon]\to X$ the respective naturals projections. Observe that the hypothesis on the Picard group of $Y$ implies that $\omega_\varepsilon$ is of degree $F^*\D$. Let us first show that there exists a deformation $\alpha_\varepsilon$ of $\alpha$ on the open set $\W=X_r^*\setminus Sing(\alpha)$ such that $\omega_\varepsilon=\overline{F}^*(\alpha_\varepsilon)$.

Let $p\in Y$ be a smooth point of $F$ such that $F(p)$ is a smooth point of $X$  and consider a trivialization of $F$ around $p$, \emph{i.e.}, open analytic neighborhoods $\U$ and $\V$ with $p\in \U$ and $F(p)\in\V$ together with  local coordinates $(x_1,\dots,x_n)$ such that $F(x_1,\dots,x_n)=(x_1,\dots,x_q)$. Suppose further that $p$ is not a singular point of $\omega$.
Let us also choose compatible trivializations $\O_X(\D)\vert_\V\simeq^\phi \O_{X}$ and $p_1^*\O_Y(F^*\D)\vert_{\U[\varepsilon]}\simeq^\psi \O_{Y[\varepsilon]}$, \emph{i.e.} such that $\psi\vert_{Y}= F^*\phi$. In these coordinates, condition 1) implies $ \omega_\varepsilon{|_\U} \wedge dx_1\wedge\dots dx_q =0$. Applying Malgrange's Theorem, \cite[Proposition (1.1), p.~67]{malgrange2}, we can get a description of $\omega_\varepsilon{|_\U}$ as
\[
\omega_\varepsilon{|_\U}=\sum_{i=1}^q h_i^\varepsilon(x)\  dx_i,
\]
for some functions $h_i^\varepsilon$ on $\U[\varepsilon]$. The integrability condition $\omega_\varepsilon\wedge d\omega_\varepsilon$ implies that  
\[0=h_i^\varepsilon \frac{\partial h_k^\varepsilon}{\partial x_j}- h_k^\varepsilon \frac{\partial h_i^\varepsilon}{\partial x_j}=\frac{\partial }{\partial x_j}\left(\frac{ h_i^\varepsilon}{ h_k^\varepsilon} \right) \hspace{0.5cm} \forall j>q \hbox{ , } \forall 1\leq i,k\leq q. \]
Since $\omega(p)\neq 0$, we may assume that $h_1^\varepsilon\in \O_{Y[\varepsilon]}^*$. Setting $k=1$ in the equation above, we get $h_i^\varepsilon=h_1^\varepsilon g_i^\varepsilon(x_1,\dots,x_q)$ for some functions $g_1^\varepsilon,\dots,g_q^\varepsilon$. 
Setting $\alpha_\varepsilon=\sum_{i=1}^q g_i^{\varepsilon} dx_i$ on $\V$, this implies
\begin{equation}\label{ecupb} h^1_\varepsilon\overline{F}^*\alpha_\varepsilon=\omega_\varepsilon.
\end{equation}
Being $F$ dominant, $\alpha_\varepsilon$ is also integrable and therefore defines a local deformation of $\alpha$. Observe further that the subsheaf $(\alpha_\varepsilon)\subseteq \Omega^1_{X[\varepsilon]|\NN}(\V)$ does not depend on the choice of $\U$: if $\U'\to \V$ is another trivialization and $\alpha_\varepsilon^1$, $\alpha_\varepsilon^2$ are the elements arising from the construction above, then on $\U''=\U\cap\U'$ we have
\[\overline{F}^*(\alpha_\varepsilon^1)=u_\varepsilon \overline{F}^*(\alpha_\varepsilon^2) \]
for some unit $u_\varepsilon$.
Since $F$ has connected fibers, we may assume that this intersection is non-empty. The equation above implies that $u=\overline{F}^*(v)$ for some unit $v$ in $F(\U'')$ and therefore both differential forms define the same submodule.
This construction yields a locally free rank one subsheaf of $\L_\varepsilon^{-1}\subseteq \Omega^1_{X_r[\varepsilon]|\NN}\vert_{\W[\varepsilon]}$, which is induced by a section $\alpha_\varepsilon \in H^0(\W[\varepsilon],\Omega^1_{X_r[\varepsilon]|\NN}\otimes \L_\varepsilon)$. 

We will now show that this deformation extends to $X$. Recall that $\omega_\varepsilon$ is of degree $F^*\D$. Moreover, since $\overline{F}^*(\alpha_\varepsilon)=\omega_\varepsilon$ on $F^{-1}(\W)$, we have 
\begin{equation} \label{iso1}\overline{F}^*\L_\varepsilon\vert_{F^{-1}(\W)[\varepsilon]}\simeq p_1^*\O_Y(F^*\D).
\end{equation}
Now let us consider the blow-up  $\pi:Y' \rightarrow Y$ along the base locus of $F$, equipped with a regular extension $F':  Y' \to X$. We will also denote by $\overline{\pi}$ and $\overline{F'}$ the respective trivial deformations. It follows from \cref{iso1} that the pullback under $\overline{F'}$ of $\L_\varepsilon$ along $(F')^{-1}(\W)$ is also a trivial deformation, \emph{i.e.},
\[
\overline{F'}^*(\L_\varepsilon)\vert_{(F')^{-1}(\W)[\varepsilon]}\simeq (\overline{\pi}^*p_1^*\O_Y(F^*\D))\vert_{(F')^{-1}(\W)[\varepsilon]}\simeq \overline{F'}^*\pi_1^*\O_X(\D)\vert_{(F')^{-1}(\W)[\varepsilon]}.
\]
Being the morphism $\overline{F'}: (F')^{-1}(\W)[\varepsilon] \to \W[\varepsilon]$ proper, by projection formula this implies that $\L_\varepsilon\simeq \pi_1^*\O_X(\D)$. This is, $\alpha_\varepsilon$ is in fact a section of $\Omega^1_{X_r[\varepsilon]|\NN}\otimes \pi_1^*\O_X(\D)$ on $\W[\varepsilon]$ of the form $\alpha_\varepsilon=\alpha + \varepsilon \eta$ for some $\eta\in H^0(\W,j_*(\Omega^1_{X_r}\otimes\O_X(\D)))$. Since $\codim(X\setminus \W)\geq 2$, we can extend $\alpha_\varepsilon$ to a degree $\D$ first order deformation of $\alpha$ on $X$ and the theorem follows.
\end{proof}

Now we shall characterize the deformations of the form $F_\varepsilon^*(\alpha)$. For this, we recall from \cite[Section 3.1, p.~1598]{moli} the definition of the module of graded projective unfoldings. Unless stated otherwise, we consider $\omega\in \F_1(\P^n,\ell)$.

\begin{definition}\label{action} We define the $S$-module of \emph{graded projective unfoldings} of $\omega$ as
\[
\UU(\omega) = \left\{(h,\eta)\in S\oplus \Omega^1_S \,:\ L_R(h)\, d\omega = L_R(\omega)\wedge(\eta - dh) \right\}\big/ S.(0,\omega).
\]
For $r\in \N$, the homogeneous component of degree $r$ can be written as
\begin{equation}\label{ecu-10}
\UU(\omega)(r) = \left\{(h,\eta)\in S_r\oplus \Omega_S^1\,(r): \ r\ h\, d\omega = \ell \, \omega\wedge(\eta - dh) \right\}\big/ S_{r-\ell}.(0,\omega).
\end{equation}
For $(h,\eta)\in \UU(\omega)(r)$ and $f\in S_{s}$, the graded \emph{$S$-module structure} is defined by
\begin{equation*}
f\cdot (h,\eta) = \left( fh,\  \tfrac{(r+s)}{r}\ f\eta \ +\ \tfrac{1}{r}\left(r\ h\ df- s\ f\ dh\right)\ \right)\in\UU(\omega)(r+s).
\end{equation*}

\end{definition}

\begin{remark}\label{rem7} By \cite[Proposition 3.2, p.~1598]{moli} every pair $(h,\eta)\in \UU(\omega)(r)$ belongs to $H^0(\P^n,(\O_{\P^n}\oplus \Omega^1_{\P^n})(r))$. Then $\UU(\omega)(\ell)$ coincides with $U(\omega)$, see \cref{defU}.
\end{remark}

\begin{definition}\label{idealI} Let $\pi_1:\UU(\omega)\to S$ be the projection to the first coordinate. We define the \emph{unfoldings ideal} associated to $\omega$ as 
\[
\begin{aligned}
I(\omega) &= \pi_1(\UU(\omega)) = \left\{ h\in S:\  h\, d\omega = \omega\wedge\widetilde{\eta}\text{ for some }\widetilde{\eta}\in \Omega_{S}^1\right\}.\\ 
\end{aligned}
\]
\end{definition}

\begin{remark}\label{rem6}
It is shown in \cite[Proposition 3.6, p.~1599]{moli} that $\pi_1$ is actually an isomorphism between the graded modules $\UU(\omega)$ and $I(\omega)$  whenever $\codim(Sing(\omega))\geq 2$. Then, the space of unfoldings of $\omega$ is completely determined by its ideal. \end{remark}

By \cite[Proposition 4.7, p.~1035]{MMQ} we get the following chain of inclusions.
\begin{proposition}\label{incl2}
With the notations above, we have $J(\omega)\subseteq I(\omega)\subseteq K(\omega)$.
\end{proposition}

Let $X$ be a simplicial complete toric surface, $\alpha\in H^0(X,j_*(\Omega^1_{X_r} \otimes \O_{X}(\D)\vert_{X_r}))$ of the form $\alpha=\sum_{i=1}^m A_i(z)dz_i $ and $F:\P^n\dashrightarrow X$ be a dominant rational map with a complete polynomial lifting of degree $\bar{e}$. In the following proposition  we compute the unfoldings ideal $I(\omega)$ for $\omega = F^{\ast}(\alpha) \in\F_1(\P^n,\ell)$ under certain generic assumptions.

\begin{lemma} With the notation above, we have 
\[ \langle A_1(F),\dots, A_m(F) \rangle \subseteq I(\omega) .\]
\end{lemma}
\begin{proof} This follows directly from \cref{cuenta}.
\end{proof}
\begin{proposition}\label{IFw}
Suppose $(F,\alpha)$ is affinely generic and $F^*:S_X\to S_{\P^n}$ is flat. Then 
\[
I(\omega)=K(\omega)=\langle A_1(F),\ldots
,A_m(F)\rangle\ .
\]
\end{proposition}

\begin{proof} Just combine \cref{kupkagen}, the lemma above and \cref{incl2}.
\end{proof}

\begin{theorem}\label{pullback}
Suppose $(F,\alpha)$ is a affinely generic and $F^*:S_X\to S_{\P^n}$ is flat. Let  $\omega_\varepsilon$ be a first order deformation of $\omega$. Then, $\omega_\varepsilon$ is induced by a first order unfolding if and only if there exist homogeneous polynomials $G_j$ of degrees $\deg(G_j)=\deg(F_j)$ such that 
\[ \omega_\varepsilon=(F+\varepsilon G)^\ast(\alpha).\]
\end{theorem}

\begin{proof}
First, observe that for any $G_1,\dots, G_m$ by straightforward calculation the first order deformation $(F+\varepsilon G)^*(\alpha)$ can be expressed as
\[
(F+\varepsilon G)^\ast(\alpha)=\omega + \varepsilon \left(\sum_{i=1}^m \sum_{j=1}^m \frac{\partial A_i}{\partial z_j}(F)G_j dF_i + \sum_{i=1}^m A_i(F) dG_i \right)\in  H^0(\PP^n[\varepsilon],\Omega^1_{\PP^n[\varepsilon]|\NN}(\ell))\ .
\]
Notice that a first order deformation $\omega_\varepsilon=\omega + \varepsilon \eta$ is induced by a first order unfolding, in the sense of \cref{unfdef}, if and only if there exists some $G \in S_{\P^n}$ of degree $\ell$ such that $(G,\eta) \in U(\omega)= \UU(\omega)(\ell)$. Recall from \cref{rem6} that $\UU(\omega)\simeq I(\omega)$ as graded $S_{\P^n}$-modules. Since $(F,\alpha)$ is affinely generic, by \cref{IFw} we have $ I(\omega) = \langle A_k(F) \rangle_{k=1}^m$.

We proceed to find  $\eta_k$ such that $(A_k(F),\eta_k)\in \UU(\ell-e_k)$.
By \cref{cuenta}, if
\[
\widetilde{\eta}_k = F^*(\ii_{\partial/\partial z_k}d\alpha)=  \sum_{j\neq k} \left( \frac{\partial A_j}{\partial z_k}(F)-\frac{\partial A_k}{\partial z_j}(F)\right)\ dF_j\ ,
\]
then we have $A_k(F) d\omega= \omega\wedge \widetilde{\eta}_k$. Now we need to determine $\eta_k$ satisfying the equation
\[
(\ell-e_k) \ A_k(F) \ d\omega = \ell \ \omega\wedge (\eta_k-dA_k(F))\ .
\]
Thereby, the relation between $\widetilde{\eta}_k$ and $\eta_k$ is given by $\eta_k = \frac{\ell-e_k}{\ell}\ \widetilde{\eta}_k+dA_k(F)$. This implies 
\[
\begin{aligned}
\eta_k &= \frac{\ell-e_k}{\ell}\ \sum_{j\neq k} \left(\frac{\partial A_j}{\partial z_k}(F)-\frac{\partial A_k}{\partial z_j}(F)\right)\ dF_j+dA_k(F) \ .
\end{aligned}
\]
Since $\UU(\omega)\simeq I(\omega)$ as graded $S_{\P^n}$-modules, it follows that $\omega_\varepsilon$ is induced by a first order unfolding if and only if
\[
\eta = \sum_{k=1}^m \pi_2(G_k\cdot (A_k(F), \eta_k)) 
\]
for some $G_k \in S_{\P^n}$ of degree $e_k$. Finally, using the module structure of \cref{action} the following calculation implies our claim:
\begin{align*}
\pi_2(G_k\cdot(A_k(F), \eta_k)) = \frac{\ell}{\ell-e_k}G_k \eta_k + \frac{1}{\ell-e_k}\left((\ell-e_k)A_k(F) dG_k - e_k G_k dA_k(F)\right) =& \\
= \frac{\ell}{\ell-e_k}G_k \left(\frac{\ell-e_k}{\ell}\ \eta_k+dA_k(F)\right)+ \frac{1}{\ell-e_k}\left((\ell-e_k)A_k(F) dG_k - e_k G_k dA_k(F)\right) = &\\
= G_k\eta_k+G_kdA_k(F)+A_k(F)dG_k= \sum_{j=1}^m  \frac{\partial A_j}{\partial z_k}(F) G_k \ dF_j + A_k(F) dG_k \ .&
\end{align*}
\end{proof}

We end this section by describing some properties of $I(\omega)$ under less restrictive assumptions on the pair $(F,\alpha)$.  
We refer the reader to \cite[Remark 3.4, p.~8]{MMQ2} for a sheaf-theoretic version of the following definition.

\begin{definition}\label{unftorico}
Let $X$ be a simplicial complete toric variety and $\alpha\in \F_1(X,\D)$. We  define the \emph{unfoldings ideal} of $\alpha$ as
\[
I(\alpha) = (\alpha\wedge \,\Omega_{S}^1:d\alpha)=\{h\in S: hd\alpha = \alpha \wedge\widetilde{\eta}, \text{ for some }\widetilde{\eta}\in \,\Omega_{S}^1\}\ .
\]
\end{definition}
\begin{remark}\label{remunfoldingsaturado}
As in \cref{incl2}, we have that $J(\alpha)\subset I(\alpha)\subset K(\alpha)$. Consequently, straightforward computation shows
\begin{equation}\label{inclusions}
J(\omega) \subset F^\ast(J(\alpha)) \subset F^\ast (I(\alpha)) \subset I(\omega) \subset K(\omega). 
\end{equation}
\end{remark}
Now, we introduce a slight variant of \cite[Definition 4.8, p.~1035]{MMQ}.

\begin{definition}
We say that a prime ideal $\mathfrak{p}\subset S$ is a \emph{division point} of $\alpha$ if $1\in I(\alpha)_{\mathfrak{p}}$.
\end{definition}

These points are exactly the primes $\mathfrak{p}$ around which $d\omega_\mathfrak{p}$ is divisible by $\omega_\mathfrak{p}$, \emph{i.e.}, there exists a local $1$ form $\eta$ around $\mathfrak{p}$ such that  $d\omega_\mathfrak{p}=\omega_\mathfrak{p} \wedge \eta $. Following \cite[Remark 4.11, p.~1035]{MMQ} it is reasonable to assume that every $\mathfrak{p}$ outside $\K(\omega)$ is a division point of $\omega$.
The next result provides further information on the inclusions in \cref{inclusions} under weaker hypotheses on the differential form  $\alpha$.
\begin{proposition}\label{propultima}
Let $X$ be a simplicial complete toric surface. With the notation above, suppose the pair $(F,\alpha)$ is generic, $F^*:S_X\to S_{\P^n}$ is flat and $I_{ij}$ is a  division point of $\alpha$ for  every $(i,j) \in \Gamma_{\alpha}\backslash \Gamma_{\alpha,\K}$.   Then,
\begin{enumerate}
\item[1)]  $(I(\omega):F^\ast(I_Z)^\infty) =  (K(\omega):F^\ast(I_Z)^\infty)$ defines the same closed  subscheme as $(F^\ast(J(\alpha)):F^\ast(I_Z)^\infty)$ and
\item[2)] $\sqrt{F^\ast(I(\alpha))} = \sqrt{ I(\omega)} = \sqrt{ K(\omega)}$.
\end{enumerate}
\end{proposition}

\begin{proof}
Since $(F,\alpha)$ is generic ,  we have  $(J(\alpha):I_Z^{\infty}) = (I(\alpha):I_Z^{\infty}) = (K(\alpha):I_Z^{\infty})$. 
Being $F^\ast$  flat and using again
\cite[Chap. 1, 2, Remarque, p.~41]{B}, it follows that
\begin{equation}\label{inclusionespullback}
(F^{\ast} (J(\alpha)):F^{\ast}(I_Z)^{\infty}) = (F^{\ast} (I(\alpha)):F^{\ast} (I_Z)^{\infty}) = ({F^\ast} (K(\alpha)):F^{\ast} (I_Z)^{\infty}).
 \end{equation}
Let $I$ be the irrelevant ideal of $\P^n$. By \cite[Proposition 3.1, p.~863]{CAMQ} we know that $(I(\omega):I^{\infty}) = I(\omega)$.
Then by \cref{inclusions} we also get
\begin{align}\label{inclusionescocientes}
((F^\ast(J(\alpha)):F^\ast(I_Z)^{\infty}):I^{\infty})  
\subset(I(\omega):F^*(I_Z)^{\infty}) \subset (K(\omega):F^*(I_Z)^{\infty}) \subset \\ \notag \subset ((K(\omega):F^*(I_Z)^{\infty}):I^{\infty}).
\end{align}
Recall from  \cref{kupkascheme}  that  $F^\ast(K(\alpha))$ and $K(\omega)$ define the same closed subscheme. This fact combined with \cref{inclusionespullback} implies that
\[
 ((F^\ast(J(\alpha)):F^\ast(I_Z)^{\infty}):I^{\infty}) = ((K(\omega):F^*(I_Z)^{\infty}):I^{\infty}),
\]
and then our first claim follows from \cref{inclusionescocientes}.

Finally, since  $I_{ij}$ is a  division point of $\alpha$ for every $(i,j) \in \Gamma_{\alpha}\backslash \Gamma_{\alpha,\K}$, the associated primes of $I(\alpha)$ and $K(\alpha)$ containing $I_Z$ are the same. Therefore, by \cref{singcomponents} the primes in the support of $F^\ast(I(\alpha))$, $I(\omega)$  and $K(\omega)$ containing $F^\ast(I_Z)$ coincide as well, which implies our second claim.
\end{proof}

{\tiny

\begin{thebibliography}{10}

\bibitem{AD}
C.~Araujo and S.~Druel.
\newblock On {F}ano foliations.
\newblock {\em Adv. Math.}, 238:70--118, 2013.

\bibitem{AD2}
C.~Araujo and S.~Druel.
\newblock On {F}ano foliations 2.
\newblock In {\em Foliation theory in algebraic geometry}, Simons Symp., pages
  1--20. Springer, Cham, 2016.

\bibitem{atiyahmacdonald}
M.~F. Atiyah and I.~G. Macdonald.
\newblock {\em Introduction to commutative algebra}.
\newblock Addison-Wesley Publishing Co., Reading, Mass.-London-Don Mills, Ont.,
  1969.

\bibitem{CB}
V.~V. Batyrev and D.~A. Cox.
\newblock On the {H}odge structure of projective hypersurfaces in toric
  varieties.
\newblock {\em Duke Math. J.}, 75(2):293--338, 1994.

\bibitem{B}
N.~Bourbaki.
\newblock {\em \'{E}l\'{e}ments de math\'{e}matique. {F}ascicule {XXVII}.
  {A}lg\`ebre commutative. {C}hapitre 1: {M}odules plats. {C}hapitre 2:
  {L}ocalisation}.
\newblock Actualit\'{e}s Scientifiques et Industrielles, No. 1290. Herman,
  Paris, 1961.

\bibitem{BB}
G.~Brown and J.~Buczy\'{n}ski.
\newblock Maps of toric varieties in {C}ox coordinates.
\newblock {\em Fund. Math.}, 222(3):213--267, 2013.

\bibitem{CAMQ}
O.~Calvo-Andrade, A.~Molinuevo, and F.~Quallbrunn.
\newblock On the geometry of the singular locus of a codimension one foliation
  in {$\mathbb P^n$}.
\newblock {\em Rev. Mat. Iberoam.}, 35(3):857--876, 2019.

\bibitem{CLE}
D.~Cerveau, A.~Lins~Neto, and S.~J. Edixhoven.
\newblock Pull-back components of the space of holomorphic foliations on
  {${\mathbb C}{\mathbb P}(n)$}, {$n\geq 3$}.
\newblock {\em J. Algebraic Geom.}, 10(4):695--711, 2001.

\bibitem{CLNLPT}
D.~Cerveau, A.~Lins-Neto, F.~Loray, J.~V. Pereira, and F.~Touzet.
\newblock Algebraic reduction theorem for complex codimension one singular
  foliations.
\newblock {\em Comment. Math. Helv.}, 81(1):157--169, 2006.

\bibitem{CP2}
S.~C. Coutinho and J.~V. Pereira.
\newblock On the density of algebraic foliations without algebraic invariant
  sets.
\newblock {\em J. Reine Angew. Math.}, 594:117--135, 2006.

\bibitem{CLS}
D.~A. Cox, J.~B. Little, and H.~K. Schenck.
\newblock {\em Toric varieties}, volume 124 of {\em Graduate Studies in
  Mathematics}.
\newblock American Mathematical Society, Providence, RI, 2011.

\bibitem{CGM}
F.~Cukierman, J.~Gargiulo~Acea, and C.~Massri.
\newblock Stability of logarithmic differential one-forms.
\newblock {\em Trans. Amer. Math. Soc.}, 371(9):6289--6308, 2019.

\bibitem{fj}
F.~Cukierman and J.~V. Pereira.
\newblock Stability of holomorphic foliations with split tangent sheaf.
\newblock {\em Amer. J. Math.}, 130(2):413--439, 2008.

\bibitem{CPV}
F.~Cukierman, J.~V. Pereira, and I.~Vainsencher.
\newblock Stability of foliations induced by rational maps.
\newblock {\em Ann. Fac. Sci. Toulouse Math. (6)}, 18(4):685--715, 2009.

\bibitem{DOLGACHEV}
I.~Dolgachev.
\newblock Weighted projective varieties.
\newblock In {\em Group actions and vector fields ({V}ancouver, {B}.{C}.,
  1981)}, volume 956 of {\em Lecture Notes in Math.}, pages 34--71. Springer,
  Berlin, 1982.

\bibitem{G}
A.~Grothendieck.
\newblock {\'El\'ements de g\'eom\'etrie alg\'ebrique. IV: \'Etude locale des
  sch\'emas et des morphismes de sch\'emas. (Troisi\`eme partie). R\'edig\'e
  avec la colloboration de Jean Dieudonn\'e.}
\newblock {\em {Publ. Math., Inst. Hautes \'Etud. Sci.}}, 28:1--255, 1966.

\bibitem{hartstable}
R.~Hartshorne.
\newblock Stable reflexive sheaves.
\newblock {\em Math. Ann.}, 254(2):121--176, 1980.

\bibitem{kupka}
I.~Kupka.
\newblock The singularities of integrable structurally stable {P}faffian forms.
\newblock {\em Proc. Nat. Acad. Sci. U.S.A.}, 52:1431--1432, 1964.

\bibitem{malgrange2}
B.~Malgrange.
\newblock Frobenius avec singularit\'{e}s. {II}. {L}e cas g\'{e}n\'{e}ral.
\newblock {\em Invent. Math.}, 39(1):67--89, 1977.

\bibitem{MMQ}
C.~Massri, A.~Molinuevo, and F.~Quallbrunn.
\newblock The kupka scheme and unfoldings.
\newblock {\em Asian Journal of Mathematics}, 22(6):1025--1046, 2018.

\bibitem{MMQ2}
C.~Massri, A.~Molinuevo, and F.~Quallbrunn.
\newblock Foliations with persistent singularities.
\newblock {\em Journal of Pure and Applied Algebra}, 225(6):106630, 2021.

\bibitem{medeiros}
A.~S.~de Medeiros.
\newblock Structural stability of integrable differential forms.
\newblock In {\em Geometry and topology}, pages 395--428. Springer, 1977.

\bibitem{moli}
A.~Molinuevo.
\newblock Unfoldings and deformations of rational and logarithmic foliations.
\newblock {\em Ann. Inst. Fourier (Grenoble)}, 66(4):1583--1613, 2016.

\bibitem{M}
D.~Mumford.
\newblock {\em The red book of varieties and schemes}, volume 1358 of {\em
  Lecture Notes in Mathematics}.
\newblock Springer-Verlag, Berlin, expanded edition, 1999.
\newblock Includes the Michigan lectures (1974) on curves and their Jacobians,
  With contributions by Enrico Arbarello.

\bibitem{PC}
J.~V. Pereira and C.~Spicer.
\newblock Hypersurfaces quasi-invariant by codimension one foliations.
\newblock {\em Math. Ann.}, 378(1-2):613--635, 2020.

\bibitem{S}
R.~P. Stanley.
\newblock {\em Combinatorics and commutative algebra}, volume~41 of {\em
  Progress in Mathematics}.
\newblock Birkh\"{a}user Boston, Inc., Boston, MA, second edition, 1996.

\bibitem{suwa}
T.~Suwa.
\newblock Unfoldings of complex analytic foliations with singularities.
\newblock {\em Japan. J. Math. (N.S.)}, 9(1):181--206, 1983.

\bibitem{suwa-multiform}
T.~Suwa.
\newblock Unfoldings of foliations with multiform first integrals.
\newblock {\em Ann. Inst. Fourier (Grenoble)}, 33(3):99--112, 1983.

\bibitem{suwa-meromorphic}
T.~Suwa.
\newblock Unfoldings of meromorphic functions.
\newblock {\em Math. Ann.}, 262(2):215--224, 1983.

\bibitem{V}
S.~Velazquez.
\newblock Toric foliations with split tangent sheaf.
\newblock {\em Bull. Sci. Math.}, 175:Paper No. 103099, 23, 2022.

\end{thebibliography}

}

\

{\tiny
\noindent
\begin{tabular}{l l}
	Javier Gargiulo Acea$^*$ \hspace{2cm}\null&\textsf{jngargiulo@gmail.com}\\
	Ariel Molinuevo$^\dag$  &\textsf{arielmolinuevo@gmail.com}\\
	Sebasti\'an Velazquez$^\ddag$  &\textsf{velaz.sebastian@gmail.com}\\
\end{tabular}}

\

\

{\tiny
\noindent
\begin{tabular}{l l l}
	$^*${IMPA} & \hspace{1cm}$^\dag$Instituto de Matemática& $^\ddag$ {Departamento de Matem\'atica}  \\
	{Estrada Dona Castorina 110} &\hspace{1cm}Av. Athos da Silveira Ramos 149 & {Pabell\'on I, Ciudad Universitaria} \\
	{Jardim Botanico} &\hspace{1cm}Bloco C, Centro de Tecnologia, UFRJ  & {FCEyN, UBA}\\
	{CEP 22460-320 } &\hspace{1cm}Cidade Universitária, Ilha do Fund\~ao & {CP C1428EGA} \\ 
	{Rio de Janeiro, RJ}&\hspace{1cm}CEP 21941-909  &  {Buenos Aires}  \\
	{Brasil}  &\hspace{1cm}Rio de Janeiro, RJ &{Argentina}   \\
	&\hspace{1cm}Brasil  \\
\end{tabular}}

\end{document}